\PassOptionsToPackage{backref=page}{hyperref}
\documentclass[11pt,letterpaper,reqno]{amsart}

\usepackage[T1]{fontenc}
\usepackage{amsmath,amssymb,amsfonts,amsthm}
\usepackage{mathtools}
\usepackage{aliascnt}
\usepackage{microtype}
\usepackage{xcolor}
\usepackage{doi}
\usepackage{hyperref}
\usepackage{bookmark}
\usepackage[capitalize,noabbrev]{cleveref}

\hypersetup{
  pdfstartview={FitH},
  colorlinks=true,
  linkcolor=blue!55!black,
  citecolor=green!35!black,
  urlcolor=blue!55!black
}
\renewcommand*{\backref}[1]{}
\renewcommand*{\backrefalt}[4]{%
  \ifcase #1
  \or
    \quad$\hookleftarrow$, cited on page~#2%
  \else
    \quad$\hookleftarrow$, cited on pages~#2%
  \fi
}

\newtheorem{theorem}{Theorem}[section]

\newaliascnt{lemma}{theorem}
\newtheorem{lemma}[lemma]{Lemma}
\aliascntresetthe{lemma}

\newaliascnt{proposition}{theorem}
\newtheorem{proposition}[proposition]{Proposition}
\aliascntresetthe{proposition}

\newaliascnt{corollary}{theorem}
\newtheorem{corollary}[corollary]{Corollary}
\aliascntresetthe{corollary}

\newaliascnt{conjecture}{theorem}

\aliascntresetthe{conjecture}

\newaliascnt{problem}{theorem}

\aliascntresetthe{problem}

\theoremstyle{definition}
\newaliascnt{definition}{theorem}

\aliascntresetthe{definition}

\newaliascnt{example}{theorem}

\aliascntresetthe{example}

\newaliascnt{remark}{theorem}
\newtheorem{remark}[remark]{Remark}
\aliascntresetthe{remark}

\crefname{theorem}{Theorem}{Theorems}
\Crefname{theorem}{Theorem}{Theorems}
\crefname{lemma}{Lemma}{Lemmas}
\Crefname{lemma}{Lemma}{Lemmas}
\crefname{proposition}{Proposition}{Propositions}
\Crefname{proposition}{Proposition}{Propositions}
\crefname{corollary}{Corollary}{Corollaries}
\Crefname{corollary}{Corollary}{Corollaries}
\crefname{conjecture}{Conjecture}{Conjectures}
\Crefname{conjecture}{Conjecture}{Conjectures}

\DeclareMathOperator{\sech}{sech}
\newcommand{\C}{\mathbb C}
\newcommand{\D}{\mathbb D}
\newcommand{\E}{\mathbb E}
\newcommand{\R}{\mathbb R}
\newcommand{\dd}{\,\mathrm d}
\newcommand{\ee}{\mathrm{e}}
\newcommand{\ii}{\mathrm{i}}
\newcommand{\mD}{\mathrm m_{\D}}

\begin{document}
\title[Integral Representations for Areal Mahler Measures]
{Integral Representations and Asymptotics for a Family of Areal Mahler Measures}

\author[Q.~Tang]{Quanyu Tang}
\address{School of Mathematical Sciences, University of Science and Technology of China, Hefei 230026, P. R. China}
\thanks{Corresponding author: Quanyu Tang
(\texttt{tangquanyu827@gmail.com}).}
\email{tangquanyu827@gmail.com}

\author[S.~Zhang]{Shu Zhang}
\address{School of Mathematics and Statistics, Xi'an Jiaotong University, Xi'an 710049, P. R. China}
\email{2213715573@stu.xjtu.edu.cn}

\subjclass[2020]{Primary 11R06; Secondary 41A60, 42A38, 60E10}

\keywords{areal Mahler measure, Cayley transform, Mellin transform, Fourier transform, cyclotomic multiple polylogarithm}

\begin{abstract}
We study a problem posed by Matilde Lal\'in concerning the areal Mahler measures of the multivariable polynomial family
$$
 P_m(x_1,\ldots,x_m,u)
 =
 \prod_{j=1}^m(1+x_j)
 +
 u\prod_{j=1}^m(1-x_j),
 \qquad m\geq1.
$$
Using a probabilistic reformulation, we derive convolution and one-dimensional Fourier integral representations for $\mathrm m_{\mathbb D}(P_m)$. We prove that, for every fixed $m$, the value $\pi^m\mathrm m_{\mathbb D}(P_m)$ belongs to the algebra of level-$4$ cyclotomic multiple polylogarithm values, and we evaluate the first nontrivial case $m=2$ explicitly in terms of $\operatorname{Li}_4(1/2)$, $\zeta(3)$, Catalan's constant, $\pi$, and $\log2$. We also derive an explicit three-term asymptotic expansion for $\mathrm m_{\mathbb D}(P_m)$ as $m\to\infty$.
\end{abstract}

\maketitle

\section{Introduction}

For a nonzero Laurent polynomial
$F\in\C[z_1^{\pm1},\ldots,z_n^{\pm1}]$, its \emph{logarithmic Mahler measure} is
\begin{equation}\label{eq:classical-measure}
 \mathrm m(F)=\frac{1}{(2\pi)^n}
 \int_{[0,2\pi]^n}
 \log\bigl|F(\ee^{\ii\theta_1},\ldots,\ee^{\ii\theta_n})\bigr|
 \dd\theta_1\cdots\dd\theta_n.
\end{equation}
In one variable, Jensen's formula \cite{Jensen1899} expresses
\eqref{eq:classical-measure} in terms of the roots of the polynomial.
Lehmer's search for monic integral polynomials of small positive Mahler
measure \cite{Lehmer1933} led to one of the central open problems in the
subject.  Mahler introduced the several-variable form in his work on
polynomial inequalities \cite{Mahler1962}.  Smyth's evaluations
\cite{Smyth1981}, followed by Boyd's extensive numerical and conjectural
investigations \cite{Boyd1981,Boyd1998}, revealed persistent connections
with special values of zeta and $L$-functions.  Regulator interpretations
were developed by Deninger \cite{Deninger1997} and Rodriguez Villegas
\cite{RodriguezVillegas1999}.  For modern accounts of the subject, see
Brunault and Zudilin \cite{BrunaultZudilin2020} and McKee and Smyth
\cite{McKeeSmyth2021}.  Higher Mahler measures and zeta Mahler functions
provide complementary moment-theoretic viewpoints
\cite{Akatsuka2009,KurokawaLalinOchiai2008,Ringeling2022}.

Let $\D=\{z\in\C:|z|<1\}$, and let $\dd A$ denote planar Lebesgue
measure.  Pritsker \cite{Pritsker2008}, motivated by questions in Bergman
spaces, introduced the logarithmic areal Mahler measure.  For nonzero
polynomials $F,G\in\C[z_1,\ldots,z_n]$, we use the extension
\begin{equation}\label{eq:areal-measure}
 \mD(F/G)
 =\frac{1}{\pi^n}\int_{\D^n}
 \log\left|\frac{F(z_1,\ldots,z_n)}{G(z_1,\ldots,z_n)}\right|
 \dd A(z_1)\cdots\dd A(z_n)
 =\mD(F)-\mD(G).
\end{equation}
This definition is unambiguous.  Indeed, if $P$ is a nonzero polynomial, then
$\lvert\log|P|\rvert$ is locally integrable on a neighborhood of the
closed polydisk.  Thus both
polynomial terms in \eqref{eq:areal-measure} are integrable, and
additivity makes the value independent of the chosen quotient
representation.  Pritsker's root formula is the areal analogue
of Jensen's formula.  Later work includes inequalities and arithmetic
questions \cite{ChoiSamuels2012,Flammang2015}, explicit multivariable and
higher-measure evaluations \cite{LalinRoy2024a}, the effect of power
changes of variables \cite{LalinRoy2024b}, and probabilistic methods for
areal measures \cite{LalinNairRingelingRoy2025}.  Kelley has recently
placed the construction in an adelic height-theoretic setting
\cite{Kelley2026}.

Problem~5 in the MM(P) problem list \cite{MMP2023}, posed by Matilde
Lal\'in, asks for the areal Mahler measure of the rational family
$R_m$, or equivalently of the polynomial family $P_m$, where
\begin{equation}\label{eq:families}
 \begin{aligned}
 R_m(x_1,\ldots,x_m,u)
   &=1+u\prod_{j=1}^m\frac{1-x_j}{1+x_j},\\
 P_m(x_1,\ldots,x_m,u)
   &=\prod_{j=1}^m(1+x_j)
     +u\prod_{j=1}^m(1-x_j),
 \end{aligned}
 \qquad m\geq1.
\end{equation}
For $m=1$, the polynomial $P_1$ in \eqref{eq:families} is Boyd's
$1+x+u-xu$ \cite{Boyd1992}.  Lal\'in evaluated the classical Mahler
measure for the full family in terms of special values of the Riemann
zeta function and a Dirichlet $L$-series \cite{Lalin2006}.  Her earlier
work had expressed related Mahler measures as multiple polylogarithms
\cite{Lalin2003}.  Related developments include an invariant change of
variables \cite{LalinNair2023}, higher Mahler measures
\cite{LalinLechasseur2016}, and nonlinear-degree extensions
\cite{LalinNairRoy2024}.  Nair \cite{Nair2025} has obtained an exact formula for the classical
Mahler measure of another family with arbitrarily many variables.  A nearby example is
$(1+x)(1+y)+z$: Brunault proved a conjecture of Boyd and Rodriguez
Villegas relating its classical Mahler measure to the $L$-function of the
elliptic curve of conductor~$15$ \cite{Brunault2023}, while an areal
deformation is studied
in \cite{LalinNairRingelingRoy2025}.  Lal\'in and Roy evaluated the case
$m=1$ of \eqref{eq:families} in the areal setting
\cite[Theorem~1.4]{LalinRoy2024a}.  We are not aware of a published
special-value formula for any $m\geq2$.

Our approach is probabilistic.  Let $X$ be distributed according to
normalized area measure on $\D$, and set
\[
 q(z):=\frac{1-z}{1+z},
 \qquad
 Y:=\log|q(X)|.
\]
Let $\rho$ denote the probability density of $Y$.  We determine $\rho$ and
the Mellin transform of $|q(X)|$.  If
$\psi(z)=\Gamma'(z)/\Gamma(z)$ is the digamma function, define
\begin{equation}\label{eq:phi-definition}
 \varphi(t):=\sech\!\left(\frac{\pi t}{2}\right)
 \left[
 1+\frac{t^2}{2}\operatorname{Re}\!\left(
 \psi\!\left(1+\frac{\ii t}{4}\right)
 -\psi\!\left(\frac12+\frac{\ii t}{4}\right)
 \right)
 \right].
\end{equation}
The function $\varphi$ is the characteristic function of $Y$.  The first
main result reduces the defining $(m+1)$-fold complex integral to a single
real integral.

\begin{theorem}\label{thm:fourier-main}
For every integer $m\geq1$,
\begin{equation}\label{eq:fourier-main}
 \mD(P_m)=\mD(R_m)
 =\frac{4}{\pi}\int_0^{+\infty}
 \frac{1-\varphi(t)^m}{t^2(t^2+4)}\,\dd t.
\end{equation}
\end{theorem}

The convolution formula can also be used in the special-value direction.
Let $\mathcal Z_4$ denote the $\mathbb Q(\ii)$-algebra generated by the
shuffle-regularized Goncharov polylogarithm values
\[
 G(a_1,\ldots,a_r;1),
 \qquad a_j\in\{0,1,-1,\ii,-\ii\},
\]
with the conventions specified in Section~\ref{sec:cyclotomic}.
Thus $\mathcal Z_4$ is the algebra of level-$4$ cyclotomic multiple
polylogarithm values.  The next theorem identifies a uniform period
space for the whole family.

\begin{theorem}\label{thm:cyclotomic-main}
For every integer $m\geq1$,
\begin{equation}\label{eq:cyclotomic-main}
 \pi^m\mD(P_m)\in\mathcal Z_4.
\end{equation}
More precisely, if
$F_m(r):=\rho^{*m}(-\log r)$ for $0<r<1$, where $\rho^{*m}$ denotes
the $m$-fold convolution of $\rho$, then $\pi^mF_m(r)$ is a
finite hyperlogarithmic expression with alphabet
$\{0,1,-1,\ii,-\ii\}$ and coefficients in $\mathcal Z_4$, and the
recursion \eqref{eq:F-recursion} computes it effectively for every
fixed $m$.
\end{theorem}

The first nontrivial member of the family admits a considerably simpler
closed form.  Let
\[
 \mathsf G:=\sum_{k=0}^{\infty}\frac{(-1)^k}{(2k+1)^2}
\]
denote Catalan's constant. Here and below, $\zeta$ denotes the Riemann zeta function.

\begin{theorem}\label{thm:two-factor-main}
We have
\begin{equation}\label{eq:two-factor-main}
 \begin{aligned}
 \mD(P_2)={}&2-2\log2-\frac43\log^22-\frac{91\pi^2}{90}\\
 &+\frac{1}{\pi^2}\left(
 128\operatorname{Li}_4\!\left(\frac12\right)
 +\frac{16}{3}\log^42+64\mathsf G^2-64\mathsf G
 +35\zeta(3)-2
 \right).
 \end{aligned}
\end{equation}
\end{theorem}

The proof of Theorem~\ref{thm:two-factor-main} uses three finite
computer-assisted checks: the derivative of a fixed hyperlogarithmic
primitive, its regularized endpoint evaluation, and a weight-$4$ identity
among level-$4$ Goncharov polylogarithm values.  The verifiers use
exact arithmetic and perform neither floating-point evaluation nor
integer-relation recognition.

Theorem~\ref{thm:cyclotomic-main} gives a constructive finite
multiple polylogarithm expression for each fixed $m$, while
Theorem~\ref{thm:two-factor-main} shows that the level-$4$ expression
for $m=2$ collapses to ordinary polylogarithms and familiar constants.
The density of $Y$ also yields the convolution formula in
Theorem~\ref{thm:convolution}.  Positivity of the characteristic function
then gives a continuous interpolation in the factor parameter whose derivatives alternate
in sign, as well as the strict sign pattern for all forward differences
of the sequence $\mD(P_m)$; see Corollary~\ref{cor:finite-differences}.

For the asymptotic behavior, let $\zeta$ denote the Riemann zeta function
and set
\begin{equation}\label{eq:sigma-kappa}
 \sigma:=\left(\frac{\pi^2}{4}-2\log2\right)^{1/2}>0,
 \qquad
 \kappa_4:=\frac{\pi^4}{8}-12(\log2)^2-\frac92\zeta(3).
\end{equation}
The variance and fourth cumulant of $Y$ are $\sigma^2$ and $\kappa_4$,
respectively.

\begin{theorem}\label{thm:asymptotic-main}
As $m\to\infty$ through the positive integers,
\begin{equation}\label{eq:asymptotic-main}
 \mD(P_m)
 =\frac{\sigma}{\sqrt{2\pi}}\sqrt m-\frac14
 +\frac{1}{\sqrt{2\pi m}}
 \left(\frac{1}{4\sigma}-\frac{\kappa_4}{24\sigma^3}\right)
 +o\!\left(m^{-1/2}\right).
\end{equation}
\end{theorem}

\subsection*{Paper organization}

In Section~\ref{sec:cayley}, we determine the density and characteristic
function of $Y=\log|q(X)|$, the Mellin transform of $|q(X)|$, and the
cumulants of $Y$ used later.  In Section~\ref{sec:representations}, we
express the areal Mahler measures in terms of sums of independent copies
of $Y$ and derive the convolution and Fourier integral representations.
In Section~\ref{sec:cyclotomic}, we prove the cyclotomic closure theorem
and carry out an exact hyperlogarithmic reduction of the case $m=2$.
In Section~\ref{sec:variation}, we use the Fourier formula to construct a
continuous interpolation in the number of factors, establish the
alternating-sign results for derivatives and forward differences, prove
strict subadditivity, and derive the large-$m$ asymptotic expansion.

\section{The Cayley-transform distribution}\label{sec:cayley}

The map $q(z)=(1-z)/(1+z)$ sends $\D$ conformally onto the right half-plane. Throughout this section, $X$ is uniformly distributed with respect to normalized area measure on $\D$, that is,
\[
 \mathbb P(X\in E)=\frac{1}{\pi}\int_E \dd A(z)
\]
for every measurable set $E\subseteq\D$, and $Y=\log|q(X)|$.

\subsection{Density of the logarithmic modulus}

The change of variables from the disk to the right half-plane gives the
law of $Y$ explicitly.

\begin{proposition}\label{prop:density}
The random variable $Y$ has the even density
\begin{equation}\label{eq:density}
 \begin{aligned}
 \rho(y)
 &=\frac{1}{\pi}\int_{-\pi/2}^{\pi/2}
 \frac{\dd\theta}{(\cosh y+\cos\theta)^2}\\
 &=
 \begin{cases}
 \displaystyle
 \frac{2\cosh y}{\pi\sinh^2|y|}\left(
 \frac{\arctan(\sinh|y|)}{\sinh|y|}
 -\frac{1}{\cosh^2y}
 \right),&y\neq0,\\[4mm]
 \displaystyle \frac{4}{3\pi},&y=0.
 \end{cases}
 \end{aligned}
\end{equation}
\end{proposition}

\begin{proof}
Write $w=q(z)$. The inverse map is
\[
 z=q^{-1}(w)=\frac{1-w}{1+w},
 \qquad \operatorname{Re}w>0,
\]
and its area Jacobian is $4/|1+w|^4$. In polar coordinates,
$w=r\ee^{\ii\theta}$, where $r>0$ and
$-\pi/2<\theta<\pi/2$. The normalized area measure on $\D$ becomes
\[
 \frac{1}{\pi}\dd A(z)
 =
 \frac{4}{\pi|1+w|^4}\dd A(w)
 =
 \frac{4r}{\pi(1+r^2+2r\cos\theta)^2}
 \dd r\dd\theta.
\]
Setting $y=\log r$, we obtain
\[
 \frac{1}{\pi}\dd A(z)
 =
 \frac{1}{\pi(\cosh y+\cos\theta)^2}
 \dd y\dd\theta.
\]
Integrating with respect to $\theta$ proves the first expression in
\eqref{eq:density}.

When $y\neq 0$, put
\[
 H(a)=\int_{-\pi/2}^{\pi/2}\frac{\dd\theta}{a+\cos\theta}=\frac{2\arctan\sqrt{a^2-1}}{\sqrt{a^2-1}}\qquad(a>1).
\]
After differentiating with respect to $a$,
$$H'(a)=-\int_{-\pi/2}^{\pi/2}\frac{\dd\theta}{\left(a+\cos\theta\right)^2}=2\frac{\sqrt{a^2-1}-a^2\arctan\sqrt{a^2-1}}{a\left(a^2-1\right)^{3/2}}.$$
Setting $a=\cosh y$ gives the second expression in \eqref{eq:density} when
$y\neq0$.

When $y=0$, the integral expression gives
\[
 \rho(0)=\frac{1}{\pi}\int_{-\pi/2}^{\pi/2}
 \frac{\dd\theta}{(1+\cos\theta)^2}=\frac{4}{3\pi}.
\]
The formula is even in $y$.
\end{proof}

The first expression in \eqref{eq:density} immediately implies
\[
 0<\rho(y)\leq\sech^2y\qquad(y\in\R).
\]
In particular, $Y$ has finite absolute moments of every order.

\subsection{Mellin transform and characteristic function}

The Mellin transform of $|q(X)|$ is an areal zeta Mahler function in the
sense used in \cite{LalinRoy2024a}.

\begin{proposition}\label{prop:zeta-cayley}
For $|\operatorname{Re}s|<2$, define
\[
 Z_{\D}(s;q):=\frac{1}{\pi}\int_{\D}|q(z)|^s\,\dd A(z).
\]
Then
\begin{equation}\label{eq:zeta-cayley}
 Z_{\D}(s;q)=\sec\!\left(\frac{\pi s}{2}\right)
 \left\{
 1-\frac{s^2}{4}\left[
 \psi\!\left(1+\frac{s}{4}\right)
 +\psi\!\left(1-\frac{s}{4}\right)
 -\psi\!\left(\frac12+\frac{s}{4}\right)
 -\psi\!\left(\frac12-\frac{s}{4}\right)
 \right]
 \right\}.
\end{equation}
The apparent singularities of the right-hand side at $s=\pm1$ are
removable.
\end{proposition}

\begin{proof}
First, the integral can be rewritten as $$Z_\D(s;q)=\frac{1}{\pi}\int_\D\ee^{s\log|q(z)|}\dd A(z)=\E\ee^{s\log|q(X)|}=\E\ee^{sY}.$$
The bound following Proposition~\ref{prop:density} shows that
$Z_{\D}(s;q)=\E\ee^{sY}$ is holomorphic on the strip
$|\operatorname{Re}s|<2$.  We first work in the smaller strip
$|\operatorname{Re}s|<1$. The change of variables used in the proof of Proposition~\ref{prop:density} gives
\[
 Z_{\D}(s;q)=\frac{4}{\pi}
 \int_{-\pi/2}^{\pi/2}\dd\theta\int_0^{+\infty}
 \frac{r^{s+1}}{(1+2r\cos\theta+r^2)^2}\dd r=\frac{8}{\pi}
 \int_0^{\pi/2}\dd\theta\int_0^{+\infty}
 \frac{r^{s+1}}{(1+2r\cos\theta+r^2)^2}\dd r.
\]
For $0<\theta\leq\pi/2$, set
\[
 \mathcal K_s(\theta):=\int_0^{+\infty}
 \frac{r^s}{1+2r\cos\theta+r^2}\,\dd r.
\]
A partial-fraction decomposition gives
\begin{align*}
    \mathcal K_s(\theta)&=\frac{1}{\ee^{\ii\theta}-\ee^{-\ii\theta}}\int_0^{+\infty}\left(\frac{\ee^{\ii\theta}r^s}{1+r\ee^{\ii\theta}}-\frac{\ee^{-\ii\theta}r^s}{1+r\ee^{-\ii\theta}}\right)\!\dd r\\
    &=\frac1{2\ii\sin\theta}\left(\ee^{-\ii s\theta}\int_0^{+\infty}\frac{(r\ee^{\ii\theta})^s}{1+r\ee^{\ii\theta}}\ee^{\ii\theta}\dd r-\ee^{\ii s\theta}\int_0^{+\infty}\frac{(r\ee^{-\ii\theta})^s}{1+r\ee^{-\ii\theta}}\ee^{-\ii\theta}\dd r\right).
\end{align*}
For $-1<\operatorname{Re}s<0$, use the principal branch
$z^s=\exp(s\operatorname{Log}z)$.  Rotating each of the rays
$\arg z=\pm\theta$ to the positive real axis crosses no pole of
$(1+z)^{-1}$.  The circular arcs at the origin and at infinity
contribute $O(\varepsilon^{\operatorname{Re}s+1})$ and
$O(R^{\operatorname{Re}s})$, respectively, and hence vanish as
$\varepsilon\to0^+$ and $R\to\infty$.  Euler's beta integral then gives
\[
 \int_0^{+\infty}
 \frac{(r\ee^{\pm\ii\theta})^s}{1+r\ee^{\pm\ii\theta}}
 \ee^{\pm\ii\theta}\,\dd r
 =\int_0^{+\infty}\frac{x^s}{1+x}\,\dd x
 =-\frac{\pi}{\sin(\pi s)}.
\]
Therefore, we have
\begin{equation}\label{eq:J-complex}
 \mathcal K_s(\theta)=
 \frac{\pi\sin(s\theta)}{\sin(\pi s)\sin\theta}.
\end{equation}
Both sides of \eqref{eq:J-complex} are holomorphic for
$|\operatorname{Re}s|<1$, after removing the singularity at $s=0$ on the right. Therefore, the equality \eqref{eq:J-complex} holds on the whole strip $|\operatorname{Re}s|<1$.  By differentiating with respect to $\theta$,
$$\mathcal K_s'(\theta)=\int_0^{+\infty}\frac{2r^{s+1}\sin\theta}{\left(
1+2r\cos\theta+r^2\right)^2}\dd r=\frac{\pi}{\sin(\pi s)}\frac{\dd}{\dd\theta}\left(\frac{\sin(s\theta)}{\sin\theta}\right).$$
Thus we obtain
\[
 Z_{\D}(s;q)
 =\frac{4}{\sin(\pi s)}
 \int_0^{\pi/2}\frac{1}{\sin\theta}
 \frac{\dd}{\dd\theta}\left(
 \frac{\sin(s\theta)}{\sin\theta}
 \right)\dd\theta.
\]

The endpoint at $\theta=0$ requires a cancellation.  For
$0<\varepsilon<\pi/2$, consider
\[
 I_\varepsilon(s):=\int_\varepsilon^{\pi/2}
 \frac{1}{\sin\theta}
 \frac{\dd}{\dd\theta}\left(
 \frac{\sin(s\theta)}{\sin\theta}
 \right)\dd\theta.
\]
Applying integration by parts three times yields
\[
 I_\varepsilon(s)
 =\frac12\sin\!\left(\frac{\pi s}{2}\right)
   -\frac12\sin(s\varepsilon)\csc^2\varepsilon
   +\frac{s}{2}\cos(s\varepsilon)\cot\varepsilon-\frac{s^2}{2}\int_\varepsilon^{\pi/2}
       \sin(s\theta)\cot\theta\,\dd\theta.
 \]
As $\varepsilon\to0^+$,
\[
 \sin(s\varepsilon)\csc^2\varepsilon
 =\frac{s}{\varepsilon}+O(\varepsilon),
 \qquad
 s\cos(s\varepsilon)\cot\varepsilon
 =\frac{s}{\varepsilon}+O(\varepsilon),
\]
locally uniformly in $s$.  The singular terms therefore cancel. It follows that
$$\int_0^{\pi/2}\frac{1}{\sin\theta}
 \frac{\dd}{\dd\theta}\left(
 \frac{\sin(s\theta)}{\sin\theta}
 \right)\dd\theta=\frac12\sin\!\left(\frac{\pi s}{2}\right)-\frac{s^2}{2}\int_0^{\pi/2}\sin(s\theta)\cot\theta\,\dd\theta.$$
Hence, we have
\begin{equation}\label{eq:zeta-C}
 Z_{\D}(s;q)
 =\sec\!\left(\frac{\pi s}{2}\right)
 \bigl(1-s^2\mathcal C(s)\bigr),
 \qquad
 \mathcal C(s):=\frac{1}{\sin(\pi s/2)}
 \int_0^{\pi/2}\sin(s\theta)\cot\theta\,\dd\theta.
\end{equation}

It remains to evaluate $\mathcal C(s)$.  Integration by parts gives
\[
 \mathcal{C}(s)=-\frac{s}{\sin(\pi s/2)}\int_0^{\pi/2}\cos(s\theta)\log(\sin\theta)\,\dd\theta.
\]
For $0<r<1$, consider the Fourier expansion $$\log\left|1-r\ee^{2\ii\theta}\right|=-\sum_{n=1}^\infty\frac{r^n\cos(2n\theta)}{n}.$$
We have
\begin{align*}
    \int_0^{\pi/2}\cos(s\theta)\log\left|1-r\ee^{2\ii\theta}\right|\,\dd\theta&=-\int_0^{\pi/2}\sum_{n=1}^\infty\frac{r^n\cos(2n\theta)}{n}\cos(s\theta)\dd\theta\\
    &=-\sum_{n=1}^\infty\frac{r^n}{n}\int_0^{\pi/2}\cos(2n\theta)\cos(s\theta)\dd\theta\\
    &=s\sin\!\left(\frac{\pi s}{2}\right)\sum_{n=1}^\infty\frac{(-r)^n}{n(4n^2-s^2)},
\end{align*}
where the legitimacy of termwise integration is guaranteed by the convergence of $\sum_{n=1}^\infty r^n/n$.

For $1/2\leq r<1$ and $0<\theta<\pi/2$, we have
\[
 \frac12\log2+\log(\sin\theta)
 \leq\log\left|1-r\ee^{2\ii\theta}\right|
 \leq\log2.
\]
Hence, for $s$ in a fixed compact set
$K\subset\{s\in\C:|\operatorname{Re}s|<1\}$,
\[
 \left|\cos(s\theta)
 \log\left|1-r\ee^{2\ii\theta}\right|\right|
 \leq C_K\bigl(1+|\log(\sin\theta)|\bigr),
\]
and the right-hand side is integrable on $(0,\pi/2)$.  The dominated
convergence theorem therefore gives
\begin{align*}
 \lim_{r\to1^-}\int_0^{\pi/2}\cos(s\theta)
 \log\left|1-r\ee^{2\ii\theta}\right|\,\dd\theta
 &=\int_0^{\pi/2}\cos(s\theta)\log(2\sin\theta)\,\dd\theta\\
 &=\frac{\log2}{s}\sin\!\left(\frac{\pi s}{2}\right)
   +\int_0^{\pi/2}\cos(s\theta)\log(\sin\theta)\,\dd\theta,
\end{align*}
with the quotient at $s=0$ interpreted by continuity.  The series on the
right-hand side of the preceding Fourier calculation converges locally
uniformly in $s$ on $|\operatorname{Re}s|<1$.  Hence, letting $r\to1^-$ gives
\[
 \int_0^{\pi/2}\cos(s\theta)\log(\sin\theta)\,\dd\theta
 =\sin\!\left(\frac{\pi s}{2}\right)
 \left[-\frac{\log2}{s}
 +s\sum_{n=1}^\infty\frac{(-1)^n}{n(4n^2-s^2)}\right].
\]
It follows that
\begin{equation}\label{eq:C-series}
 \mathcal C(s)
 =\log2+s^2\sum_{n=1}^\infty
 \frac{(-1)^n}{n(s^2-4n^2)}
 =4\sum_{n=1}^\infty
 \frac{(-1)^{n+1}n}{4n^2-s^2}.
\end{equation}
The first series in \eqref{eq:C-series} converges locally uniformly away
from the even integers.  The standard identity
\[
 \sum_{n=0}^\infty\frac{(-1)^n}{n+z}
 =\frac12\left[
 \psi\!\left(\frac{z+1}{2}\right)
 -\psi\!\left(\frac{z}{2}\right)
 \right]
 \qquad(\operatorname{Re}z>0),
\]
applied to the second series gives
\[
 \mathcal C(s)=\frac14\left[
 \psi\!\left(1+\frac{s}{4}\right)
 +\psi\!\left(1-\frac{s}{4}\right)
 -\psi\!\left(\frac12+\frac{s}{4}\right)
 -\psi\!\left(\frac12-\frac{s}{4}\right)
 \right].
\]
Substitution into \eqref{eq:zeta-C} proves \eqref{eq:zeta-cayley} for
$|\operatorname{Re}s|<1$.  The right-hand side is meromorphic on
$|\operatorname{Re}s|<2$, and its difference from the holomorphic
function $Z_{\D}(s;q)$ vanishes on a nonempty open subset.  Uniqueness of
meromorphic continuation proves the identity throughout the larger strip
and shows that the possible poles at $s=\pm1$ are removable.
\end{proof}

Restricting Proposition~\ref{prop:zeta-cayley} to the imaginary axis gives
the Fourier transform of the density in Proposition~\ref{prop:density}.

\begin{corollary}\label{cor:characteristic}
For every $t\in\R$,
\[
 \E\ee^{\ii tY}=\varphi(t),
\]
where $\varphi$ is defined in \eqref{eq:phi-definition}.
\end{corollary}

\begin{proof}
Setting $s=\ii t$ in \eqref{eq:zeta-cayley} gives
\begin{align*}
    \E\ee^{\ii tY}&=Z_\D(\ii t;q)\\
    &=\sec\!\left(\frac{\ii\pi t}{2}\right)\left\{1+\frac{t^2}{4}\left[\psi\!\left(1+\frac{\ii t}{4}\right)+\psi\!\left(1-\frac{\ii t}{4}\right)-\psi\!\left(\frac12+\frac{\ii t}{4}\right)-\psi\!\left(\frac12-\frac{\ii t}{4}\right)\right]\right\}\\
    &=\sech\!\left(\frac{\pi t}{2}\right)\left\{1+\frac{t^2}{2}\operatorname{Re}\left[\psi\!\left(1+\frac{\ii t}{4}\right)-\psi\!\left(\frac12+\frac{\ii t}{4}\right)\right]\right\},
\end{align*}
where the last equality follows from the identity
$\overline{\psi(z)}=\psi(\overline z)$ away from the poles.
\end{proof}

The following estimates are needed both for Fourier inversion arguments
and for differentiation with respect to the number of factors.

\begin{lemma}\label{lem:phi-estimates}
For every $t\in\R$,
\begin{equation}\label{eq:phi-bounds}
 \sech\!\left(\frac{\pi t}{2}\right)
 \leq\varphi(t)
 \leq(1+t^2\log2)\sech\!\left(\frac{\pi t}{2}\right).
\end{equation}
Moreover, $\varphi$ is even,
$0<\varphi(t)<1$ for $t\neq0$, and
$\varphi\in L^1(\R)\cap L^2(\R)$.
\end{lemma}

\begin{proof}
For $t\ne0$, setting $s=\ii t$ in equation \eqref{eq:zeta-C} gives
\begin{equation}\label{eq:phi-B}
 \varphi(t)=\sech\!\left(\frac{\pi t}{2}\right)
 \bigl(1+t^2B(t)\bigr),
 \quad
 B(t):=\frac{1}{\sinh(\pi t/2)}
 \int_0^{\pi/2}\sinh(t\theta)\cot\theta\,\dd\theta.
\end{equation}
The numerator and denominator in the definition of $B(t)$ have the same sign, so $B(t)>0$. When $t>0$ and $0\leq\theta\leq\pi/2$, convexity of the hyperbolic sine gives \(\sinh(t\theta)/\sinh(\pi t/2) \leq 2\theta/\pi\). Thus
\[
 0<B(t)\leq\frac{2}{\pi}
 \int_0^{\pi/2}\theta\cot\theta\,\dd\theta=\log2.
\]
The same conclusion holds for negative $t$ because $B$ is even. Moreover, $B$ extends continuously to $t=0$ upon setting
$B(0)=\log2$. The two
bounds in \eqref{eq:phi-bounds} now follow from \eqref{eq:phi-B}; the
upper bound implies the asserted $L^1$ and $L^2$ integrability.

The identity $\psi(\overline z)=\overline{\psi(z)}$ also shows directly
that $\varphi$ is even.  Moreover, since $\varphi(t)>0$, we have
\[
 0<\sech\!\left(\frac{\pi t}{2}\right)
 \leq \varphi(t)
 =\left|\E\ee^{\ii tY}\right|
 \leq1.
\]
If $\varphi(t)=1$, then equality in the triangle inequality would force
$\ee^{\ii tY}$ to be almost surely constant.  This is impossible for
$t\neq0$, because Proposition~\ref{prop:density} gives a strictly
positive density on all of $\R$.  Therefore
$0<\varphi(t)<1$ for $t\neq0$.
\end{proof}

The constants in \eqref{eq:sigma-kappa} can now be read off from the
transform.  The corresponding moments also follow from the higher areal
measure formulas in \cite{LalinRoy2024a}; we include the short calculation
to keep the normalization explicit.

\begin{lemma}\label{lem:cumulants}
The variance of $Y$ is $\sigma^2$, and its fourth cumulant
$\E Y^4-3(\E Y^2)^2$ is $\kappa_4$, with the constants defined in
\eqref{eq:sigma-kappa}.
\end{lemma}

\begin{proof}
Substituting $s=\ii t$ into the first series in \eqref{eq:C-series} gives
\begin{align*}
    B(t)&=\log2+t^2\sum_{n=1}^\infty\frac{(-1)^n}{n(t^2+4n^2)}\\
    &=\log2+\frac{t^2}{4}\sum_{n=1}^\infty\frac{(-1)^n}{n^3}-\frac{t^4}{4}\sum_{n=1}^\infty\frac{(-1)^n}{n^3(t^2+4n^2)}\\
    &=\log2-\frac{3\zeta(3)}{16}t^2+O(t^4).
\end{align*}
Combining this expansion with \eqref{eq:phi-B} and the Taylor series of
the hyperbolic secant yields
\begin{align*}
    \varphi(t)&=\left(1-\frac{\pi^2}{8}t^2+\frac{5\pi^4}{384}t^4+O(t^6)\right)\left(1+t^2\log2-\frac{3\zeta(3)}{16}t^4+O(t^6)\right)\\
    &=1-\frac12\left(\frac{\pi^2}{4}-2\log2\right)t^2
 +\frac{1}{24}\left(
 \frac{5\pi^4}{16}-3\pi^2\log2-\frac92\zeta(3)
 \right)t^4+O(t^6).
\end{align*}
Since $Z_{\D}(s;q)$ is holomorphic for $|\operatorname{Re}s|<2$,
the characteristic function is real analytic at the origin; in particular,
the remainder above may be differentiated to the required order.  Comparing
with the moment expansion of a characteristic function gives
\[
 \E Y=0,\qquad
 \E Y^2=\frac{\pi^2}{4}-2\log2=\sigma^2,\qquad
 \E Y^3=0,
\]
and
\[
 \E Y^4=\frac{5\pi^4}{16}-3\pi^2\log2-\frac92\zeta(3).
\]
Consequently,
\begin{align*}
 \E Y^4-3(\E Y^2)^2
 &=\frac{5\pi^4}{16}-3\pi^2\log2-\frac92\zeta(3)
   -3\left(\frac{\pi^2}{4}-2\log2\right)^2\\
 &=\frac{\pi^4}{8}-12(\log2)^2-\frac92\zeta(3)
 =\kappa_4.
\qedhere\end{align*}
\end{proof}

\section{Integral representations for the areal Mahler measure}\label{sec:representations}

We first remove two algebraically inessential features of the family in
\eqref{eq:families}.  Define
\[
 W_m(x_1,\ldots,x_m):=\prod_{j=1}^m q(x_j).
\]

\begin{lemma}\label{lem:equivalent-forms}
For every integer $m\geq1$,
\begin{equation}\label{eq:equivalent-forms}
 \mD(P_m)=\mD(R_m)=\mD(u+W_m).
\end{equation}
\end{lemma}

\begin{proof}
The factorization $P_m=R_m\prod_{j=1}^m(1+x_j)$ and additivity of the measure give
\[
 \mD(P_m)=\mD(R_m)+\sum_{j=1}^m\mD(1+x_j).
\]
Pritsker's one-variable formula \cite[Theorem~1.1]{Pritsker2008} gives
$\mD(1+x)=0$, proving the first equality in
\eqref{eq:equivalent-forms}.

We have $R_m=1+uW_m$. The simultaneous change of variables
$x_j\mapsto-x_j$ preserves normalized area measure and sends $W_m$ to
$W_m^{-1}$. Hence
\[
 \mD(R_m)=\mD(1+uW_m^{-1})
 =\mD(u+W_m)-\mD(W_m).
\]
Again by Pritsker's formula,
\[
 \mD(W_m)=\sum_{j=1}^m
 \bigl(\mD(1-x_j)-\mD(1+x_j)\bigr)=0.
\]
This proves the second equality in \eqref{eq:equivalent-forms}.
\end{proof}

\subsection{A convolution formula}

For an integrable density $f$ on $\R$, let $f^{*m}$ denote its $m$-fold convolution.

\begin{theorem}\label{thm:convolution}
For every integer $m\geq1$,
\begin{equation}\label{eq:convolution}
 \mD(P_m)=\int_0^{+\infty}
 \left(y+\frac{\ee^{-2y}-1}{2}\right)\rho^{*m}(y)\,\dd y.
\end{equation}
\end{theorem}

\begin{proof}
Let $X_1,\ldots,X_m, U$ be independent random variables uniformly distributed on $\D$. Put
\[
 Y_j:=\log|q(X_j)|,
 \quad
 S_m:=Y_1+\cdots+Y_m.
\]
Then the random variable $S_m$ has density $\rho^{*m}$. Define the random variables
\[
 \mathcal W_m:=W_m(X_1,\ldots,X_m),
 \qquad
 \mathcal R_m:=R_m(X_1,\ldots,X_m,U)
              =1+U\mathcal W_m.
\]
Since \(|\mathcal W_m| = \prod_{j=1}^m|q(X_j)| = \ee^{S_m}\), conditioning on $X_1,\ldots,X_m$ gives
\[
\begin{aligned}
 \E\left(
   \log|\mathcal R_m|
   \,\big|\,
   X_1,\ldots,X_m
 \right)
 &=
 \frac{1}{\pi}
 \int_{\D}
 \log|1+u\mathcal W_m|
 \,\dd A(u)\\
 &=
 \frac{1}{\pi}
 \int_{\D}
 \log\left|1+u|\mathcal W_m|\right|
 \,\dd A(u)\\
 &=
 \frac{1}{\pi}
 \int_{\D}
 \log|1+u\ee^{S_m}|
 \,\dd A(u),
\end{aligned}
\]
where the second equality follows from the rotational invariance of
normalized area measure on $\D$.

Pritsker's one-variable root formula yields
\begin{equation}\label{eq:h-kernel}
 h(y):=
 \frac{1}{\pi}
 \int_{\D}\log|1+\ee^y u|\,\dd A(u)
 =
 \begin{cases}
 0,&y\leq0,\\[1mm]
 \displaystyle y+\frac{\ee^{-2y}-1}{2},&y>0.
 \end{cases}
\end{equation}
Indeed, for $y\leq0$, the zero $-\ee^{-y}$ lies outside or on the
boundary of the unit disk, whereas for $y>0$ it lies inside.

The uses of conditional expectation and Fubini's theorem are justified as
follows.  The logarithm of the modulus of a nonzero polynomial is locally
integrable to every finite power on a neighborhood of the closed polydisk;
locally at a zero this reduces to
$\int_0^\varepsilon |\log t|^p\,\dd t<\infty$.  The same statement holds for
a nonzero rational function after separating numerator and denominator.
Thus all logarithmic random variables occurring here lie in $L^p$ for every
finite $p$, and H\"older's inequality gives the required absolute
integrability.  Consequently,
\[
\begin{aligned}
 \mD(R_m)
 &=
 \E\log|\mathcal R_m|\\
 &=
 \E\left[
   \E\left(
     \log|\mathcal R_m|
     \,\big|\,
     X_1,\ldots,X_m
   \right)
 \right]\\
 &=
 \E h(S_m).
\end{aligned}
\]
Using \eqref{eq:h-kernel}, the density of $S_m$, and
Lemma~\ref{lem:equivalent-forms} gives \eqref{eq:convolution}.
Moreover, since
$0\leq h(y)\leq\max\{y,0\}$ and $\rho(y)\le\sech^2y$, $$\E|h(S_m)|\le\E|S_m|\le m\E|Y_1|\le2m\int_0^{+\infty}y\sech^2y\dd y<+\infty.$$ Thus the integral is absolutely convergent.
\end{proof}

\subsection{A Fourier integral representation}

We now prove Theorem~\ref{thm:fourier-main}.

\begin{proof}
Since $\rho(y)$ is even, $\rho^{*m}(y)$ is also even. Hence the proof of Theorem~\ref{thm:convolution} gives
$$\mD(P_m)=\E h(S_m)=\int_0^{+\infty} h(y)\rho^{*m}(y)\dd y=\int_\R\frac12h(|y|)\rho^{*m}(y)\dd y.$$
Define $$g(y)=\frac12h(|y|)=\frac{|y|}{2}+\frac{\ee^{-2|y|}-1}4.$$ Then $\mD(P_m)=\E g(S_m)$.
The standard integral identities
\[
 \int_0^{+\infty}\frac{1-\cos(ty)}{t^2}\,\dd t
 =\frac{\pi|y|}{2},
 \quad
 \int_0^{+\infty}\frac{\cos(ty)}{t^2+4}\,\dd t
 =\frac{\pi}{4}\ee^{-2|y|},\quad\int_0^{+\infty}\frac1{t^2+4}\dd t=\frac\pi4
\]
imply
\begin{equation}\label{eq:g-fourier}
 g(y)=\frac{4}{\pi}\int_0^{+\infty}
 \frac{1-\cos(ty)}{t^2(t^2+4)}\,\dd t.
\end{equation}
The integrand in \eqref{eq:g-fourier} is nonnegative. Tonelli's theorem gives
\begin{align*}
    \mD(P_m)&=\frac{4}{\pi}\int_\R\rho^{*m}(y)\dd y\int_0^{+\infty}\frac{1-\cos(ty)}{t^2(t^2+4)}\dd t\\
    &=\frac4\pi\int_0^{+\infty}\frac{\dd t}{t^2(t^2+4)}\int_\R\left(1-\cos(ty)\right)\rho^{*m}(y)\dd y\\
    &=\frac4\pi\int_0^{+\infty}\frac{1-\E\cos(tS_m)}{t^2(t^2+4)}\dd t.
\end{align*}
By the independence of $Y_1,\ldots,Y_m$ and Corollary~\ref{cor:characteristic}, $$\E\cos(tS_m)=\operatorname{Re}\E\ee^{\ii tS_m}=\operatorname{Re}\left(\E\ee^{\ii tY}\right)^m=\varphi(t)^m.$$
Therefore,
\[
 \mD(P_m)=\frac{4}{\pi}\int_0^{+\infty}\frac{1-\varphi(t)^m}{t^2(t^2+4)}\,\dd t,
\]
which is \eqref{eq:fourier-main}.

Only convergence remains to be verified.  By
Lemma~\ref{lem:phi-estimates}, $0\leq\varphi(t)\leq1$. We have
\[
 0\leq1-\varphi(t)
 =\E(1-\cos(tY))
 \leq\frac{t^2}{2}\E Y^2.
\]
Since $\E Y^2<+\infty$, 
$1-\varphi(t)^m\leq m(1-\varphi(t))=O(t^2)$ near the origin.  At infinity, the integrand is $O(t^{-4})$. Thus the integral in \eqref{eq:fourier-main} is absolutely convergent.
\end{proof}

For completeness, let $\chi_{-4}$ be the primitive Dirichlet character
modulo $4$, defined by $\chi_{-4}(n)=0$ for even $n$ and
$\chi_{-4}(n)=(-1)^{(n-1)/2}$ for odd $n$, and let
$L(\chi_{-4},s)=\sum_{n\geq1}\chi_{-4}(n)n^{-s}$.
The known evaluation of the first member of the family is then recovered
without further integration.

\begin{corollary}\label{cor:one-factor}
We have
\[
 \mD(P_1)
 =\frac{6}{\pi}L(\chi_{-4},2)-\log2-\frac12-\frac1\pi.
\]
\end{corollary}

\begin{proof}
Lemma~\ref{lem:equivalent-forms} gives
$\mD(P_1)=\mD(u+q(x))$.  The stated value is exactly
\cite[Theorem~1.4]{LalinRoy2024a}.
\end{proof}

\section{Cyclotomic polylogarithms and an explicit evaluation}\label{sec:cyclotomic}

We now place the values in a cyclotomic multiple polylogarithm space and
then carry out the reduction explicitly for two factors.  Let
\[
 \mu_4:=\{1,-1,\ii,-\ii\},
 \qquad
 \Sigma_4:=\{0\}\cup\mu_4.
\]
Let $\gamma:[0,1]\to\mathbb C$ be a piecewise $C^1$ path from $0$ to $z$
whose interior lies in $\mathbb C\setminus\Sigma_4$.  For $0\leq s\leq1$,
let $\gamma_s$ denote the restriction of $\gamma$ to $[0,s]$.  We use the
path-defined Goncharov polylogarithm (GPL)
\[
 G_\gamma(a_1,\ldots,a_k;z)
 :=\int_0^1
 \frac{\gamma'(s)}{\gamma(s)-a_1}
 G_{\gamma_s}(a_2,\ldots,a_k;\gamma(s))\,\dd s,
 \qquad G_\gamma(;z):=1.
\]
The weight of $G(a_1,\ldots,a_k;z)$ is the word length $k$.
When the path is clear we omit the subscript.  For an all-zero word we use
\[
 G(\underbrace{0,\ldots,0}_{r};z)=\frac{\log^r z}{r!},
\]
with the logarithm determined by the path.  Throughout this paper, values at
$z=1$ are taken along the real segment $[0,1]$, with tangential base points
$\overrightarrow{01}$ and $\overrightarrow{10}$; endpoint-divergent words are interpreted by shuffle regularization
in the standard tangential-base-point sense; see, for example,
\cite{Au2026}.  Values at $0<z<1$ use the same
real segment.  In particular, every convergent value used below has an
unambiguous ordinary improper-integral interpretation, and
\[
 \overline{G(a_1,\ldots,a_k;1)}
 =G(\overline{a_1},\ldots,\overline{a_k};1).
\]
Writing $\omega_a=\dd t/(t-a)$, we abbreviate the corresponding iterated
integral by
\[
 G(a_1,\ldots,a_k;z)=\int_0^z\omega_{a_1}\cdots\omega_{a_k}.
\]
We shall also refer to these iterated integrals as hyperlogarithms. We denote by $\mathcal Z_4$ the $\mathbb Q(\ii)$-algebra generated by the
shuffle-regularized values
\[
 G(a_1,\ldots,a_r;1),\qquad a_j\in\Sigma_4.
\]
Equivalently, $\mathcal Z_4$ is the algebra of level-$4$ cyclotomic
multiple polylogarithm values; see, for example, \cite{Au2026}.  Notice that
$\pi\in\mathcal Z_4$, since
$G(\ii;1)-G(-\ii;1)=\ii\pi/2$.

\subsection{A cyclotomic closure theorem}

For $0<r<1$, set
\[
 \widehat\rho(r):=\rho(-\log r),
 \qquad
 F_m(r):=\rho^{*m}(-\log r).
\]
The explicit density in Proposition~\ref{prop:density} gives
\begin{equation}\label{eq:rhohat-cyclotomic}
 \begin{aligned}
 \widehat\rho(r)
 =\frac{8r^2(1+r^2)}{\pi(1-r^2)^3}
 \left(\frac{\pi}{2}-2\arctan r\right)-\frac{16r^3}{\pi(1+r^2)(1-r^2)^2}.
 \end{aligned}
\end{equation}
Although the right-hand side has apparent singularities at $r=1$, it
extends continuously there, with value $4/(3\pi)$.

\begin{lemma}\label{lem:mellin-convolution-recursion}
For every integer $m\geq1$ and every $0<r<1$,
\begin{equation}\label{eq:F-recursion}
 \begin{aligned}
 F_{m+1}(r)
 ={}&\int_0^1F_m(ru)\widehat\rho(u)\frac{\dd u}{u}
 +\int_r^1F_m(r/u)\widehat\rho(u)\frac{\dd u}{u}\\
 &+\int_0^rF_m(u/r)\widehat\rho(u)\frac{\dd u}{u}.
 \end{aligned}
\end{equation}
Moreover,
\begin{equation}\label{eq:M-cyclotomic-integral}
 \mD(P_m)
 =\int_0^1
 \left(-\log r+\frac{r^2-1}{2}\right)F_m(r)\frac{\dd r}{r}.
\end{equation}
\end{lemma}

\begin{proof}
For $y>0$, we have
\begin{align*}
    \rho^{*(m+1)}(y)&=\int_{\R}\rho^{*m}(y-z)\rho(z)\dd z\\
    &=\int_{-\infty}^0\rho^{*m}(y-z)\rho(z)\dd z+\int_0^{+\infty}\rho^{*m}(y-z)\rho(z)\dd z\\
    &=\int_0^1\rho^{*m}(y-\log u)\rho(\log u)\frac{\dd u}{u}+\int_0^1\rho^{*m}(y+\log u)\rho(-\log u)\frac{\dd u}u,
\end{align*}
where the last equality is given by making substitutions $z=\log u$ and $z=-\log u$ to the two parts of the integral, respectively. Recall that $\rho$ is even, and hence $\rho^{*m}$ is even for all $m\ge1$. Thus, setting $y=-\log r$ gives
\begin{align*}
    F_{m+1}(r)&=\int_0^1\rho^{*m}\left(-\log(ur)\right)\rho(\log u)\frac{\dd u}u+\int_0^1\rho^{*m}\left(-\log\frac r u\right)\rho(-\log u)\frac{\dd u}u\\
    &=\int_0^1F_m(ru)\widehat\rho(u)\frac{\dd u}u+\int_r^1F_m(r/u)\widehat\rho(u)\frac{\dd u}{u}+\int_0^rF_m(u/r)\widehat\rho(u)\frac{\dd u}{u}.
\end{align*}
Finally, the change
of variables $r=\ee^{-y}$ in Theorem~\ref{thm:convolution} gives
\eqref{eq:M-cyclotomic-integral}.
\end{proof}

We shall use the following elementary endpoint estimate.  Since
\[
 0\leq \rho(y)\leq \sech^2 y\leq 4\ee^{-2|y|},
\]
comparison with the convolution powers of $\ee^{-2|y|}$ gives, by
induction on $m$,
\begin{equation}\label{eq:convolution-tail-bound}
 \rho^{*m}(y)
 \leq C_m(1+|y|)^{m-1}\ee^{-2|y|}
 \qquad(y\in\mathbb R)
\end{equation}
for a constant $C_m>0$.  Consequently, \eqref{eq:convolution-tail-bound} yields
\begin{equation}\label{eq:F-endpoint-bound}
 F_m(r)
 =O_m\!\left(r^2(1+|\log r|)^{m-1}\right)
 \qquad(r\to0^+).
\end{equation}
Moreover, $F_m$ is continuous and bounded on every interval
$[r_0,1]$ with $r_0>0$.

We are now ready to present the following.

\begin{proof}[Proof of Theorem~\ref{thm:cyclotomic-main}]
Let $\mathcal H_4$ denote the $\mathcal Z_4$-algebra generated by the
rational functions in $\mathbb Q(\ii)(r)$ whose poles on $\mathbb P^1$ lie
in $\{0,\infty\}\cup\mu_4$ and by the hyperlogarithms
$G(a_1,\ldots,a_k;r)$ with $a_j\in\Sigma_4$.  Since
$\pi\in\mathcal Z_4$ and
\[
 \arctan r=\frac{G(\ii;r)-G(-\ii;r)}{2\ii},
\]
equation \eqref{eq:rhohat-cyclotomic} gives
$\pi\widehat\rho\in\mathcal H_4$.

We prove inductively that
\begin{equation}\label{eq:F-hyperlog-claim}
 \pi^mF_m\in\mathcal H_4.
\end{equation}
The assertion is clear for $m=1$.  After changing variables in the second
and third integrals of \eqref{eq:F-recursion}, we may write
\begin{align}
 \pi^{m+1}F_{m+1}(r)
 ={}&\underbrace{\int_0^1 \pi^mF_m(ru)\,\pi\widehat\rho(u)
                  \frac{\dd u}{u}}_{\mathcal I_1(r)}
   +\underbrace{\int_r^1 \pi^mF_m(u)\,\pi\widehat\rho(r/u)
                  \frac{\dd u}{u}}_{\mathcal I_2(r)}\notag\\
  &+\underbrace{\int_0^1 \pi^mF_m(u)\,\pi\widehat\rho(ru)
                  \frac{\dd u}{u}}_{\mathcal I_3(r)}.
 \label{eq:F-recursion-fibered}
\end{align}
For the three integrals in \eqref{eq:F-recursion-fibered}, we first work with
endpoint-truncated iterated integrals and with $r$ in a compact subinterval of
$(0,1)$.  Products are expanded by the shuffle identity.

For a GPL word with no trailing zero, scaling the endpoint gives
\[
 G(a_1,\ldots,a_k;ru)
 =G(r^{-1}a_1,\ldots,r^{-1}a_k;u).
\]
If a trailing zero block is present, the same substitution gives this
identity up to a polynomial in $\log r$, with coefficients GPLs of smaller
length.  Since $\log r=G(0;r)$, these additional terms already belong to
$\mathcal H_4$.  Thus the moving letters in $\mathcal I_1$ are contained in
$\{0\}\cup r^{-1}\mu_4\cup\mu_4$.  In $\mathcal I_2$, the factor
$\widehat\rho(r/u)$ is a rational combination of logarithms with letters in
$\{0\}\cup r\mu_4$, while in $\mathcal I_3$ the factor
$\widehat\rho(ru)$ has letters in
$\{0\}\cup r^{-1}\mu_4$.  No other moving singularities occur.

For completeness, we record the parameter-differentiation step, including
its lower-endpoint term.  Put
$H(t)=G(\alpha_2(r),\ldots,\alpha_k(r);t)$.  If $\alpha_1$ moves, then
\begin{align*}
 \frac{\partial}{\partial r}G(\alpha_1,\ldots,\alpha_k;u)
 ={}&-\alpha_1'(r)\frac{H(u)}{u-\alpha_1(r)}
     -\alpha_1'(r)\frac{H(0)}{\alpha_1(r)}\\
 &+\alpha_1'(r)\int_0^u\frac{H'(t)}{t-\alpha_1(r)}\,\dd t
   +\int_0^u\frac{1}{t-\alpha_1(r)}
      \frac{\partial H(t)}{\partial r}\,\dd t.
\end{align*}
Here $H(0)=0$ for a nonempty tail in tangential regularization; for a
one-letter word $H(0)=1$, and the displayed additional term is the rational
function $-\alpha_1'/\alpha_1$.  Partial fractions, followed in the
coincident-letter case by one further integration by parts, lowers the word
length.  Iterating this argument expresses the $r$-derivative as a finite
sum of shorter iterated integrals multiplied by logarithmic derivatives of
pairwise letter differences and endpoint--letter differences.  This is the
standard fibration formula; compare \cite[Lemma~2.7]{Panzer2015}.

After extracting the constant term in the endpoint cutoff parameters, an
operation that commutes with $r$-differentiation on compact subintervals of
$(0,1)$, every nonconstant $r$-dependent factor occurring in the resulting
logarithmic derivatives is, up to a nonzero constant and a power of $r$, of
one of the forms
$$
 ar-b,\qquad a-rb,\qquad 1-rb,\qquad r-a,
 \qquad a,b\in\mu_4.
$$
Unless such a factor vanishes identically, in which case it is handled by
tangential regularization, its zeros lie in $\mu_4$.  Hermite reduction of
the rational prefactors introduces no other poles.  Thus the only possible
$r$-singularities occur at $0$ and at the points of $\mu_4$.  Induction on
the GPL word length therefore shows that each regularized
$\mathcal I_j(r)$ is a hyperlogarithmic function of $r$ with alphabet
$\Sigma_4$, up to an $r$-independent constant.

A coarse complexity bound is also immediate.  The recursion increases the
largest GPL weight by at most two at each step.  Since the initial case has
weight at most one, the largest GPL weight appearing in $\pi^mF_m$ is at
most $2m-1$.

The constants of integration are fixed by letting $r\to1^-$.  The original
integrands in $\mathcal I_1$ and $\mathcal I_3$ then converge to integrable
functions on $(0,1)$ with fixed alphabet $\Sigma_4$, so their regularized
limits lie in $\mathcal Z_4$.  For $\mathcal I_2$, choose
$r_0\in(0,1)$.  Continuity of $F_m$ and $\widehat\rho$ gives a constant
$C=C(r_0,m)$ such that, for $r_0\leq r\leq u\leq1$,
\[
 \left|F_m(u)\widehat\rho(r/u)u^{-1}\right|\leq C.
\]
Consequently $\mathcal I_2(r)=O(1-r)$ as $r\to1^-$.

Finally, the regularized expressions coincide with the original absolutely
convergent convolution integrals.  Near $u=0$, estimate
\eqref{eq:F-endpoint-bound}, together with
$\widehat\rho(u)=O(u^2)$, supplies an integrable majorant uniform for $r$ on
compact subintervals of $(0,1)$.  The remaining endpoints are controlled by
continuity and the preceding shrinking-interval estimate. Dominated convergence therefore removes the truncations and proves the
induction step in \eqref{eq:F-hyperlog-claim}.

It remains to evaluate \eqref{eq:M-cyclotomic-integral}.  By
\eqref{eq:F-endpoint-bound},
\[
 \left(-\log r+\frac{r^2-1}{2}\right)\frac{F_m(r)}r
 =O_m\!\left(r(1+|\log r|)^m\right)
 \qquad(r\to0^+),
\]
whereas the same integrand is $O_m((1-r)^2)$ as $r\to1^-$. The integral is
therefore absolutely convergent.  Since $\pi^mF_m\in\mathcal H_4$,
hyperlogarithmic integration and endpoint shuffle regularization give
\[
 \pi^m\mD(P_m)
 =\int_0^1\left(-\log r+\frac{r^2-1}{2}\right)
       \pi^mF_m(r)\frac{\dd r}{r}
 \in\mathcal Z_4.
\]
This proves \eqref{eq:cyclotomic-main}, and the construction is finite and
effective for every fixed $m$.
\end{proof}

\subsection{The case of two factors}

For this subsection, write $M_j:=\mD(P_j)$, and let
\[
 \mathsf G:=\sum_{n=0}^{\infty}\frac{(-1)^n}{(2n+1)^2}
 =L(\chi_{-4},2)
\]
denote Catalan's constant.  We begin with a positive energy identity
which reduces $M_2$ to a one-dimensional hyperlogarithmic integral.

For $y\geq0$, let
\[
 T(y):=\mathbb P(Y>y)=\int_y^{+\infty}\rho(x)\,\dd x,
\]
and, for $0<r\leq1$, put
\[
 \mathcal T(r):=T(-\log r),
 \qquad
 J(r):=\int_0^r s\mathcal T(s)\,\dd s.
\]
Note that $\mathcal T(r)\to0$ as $r\to0^+$, so we can set $\mathcal T(0)=0$.
Since $\rho$ is even,
\[
 \mathcal T(1)=\int_0^{+\infty}\rho(y)\,\dd y
 =\frac12\int_{-\infty}^{+\infty}\rho(y)\,\dd y
 =\frac12.
\]

\begin{lemma}\label{lem:tail-energy}
For $0\le r<1$,
\begin{equation}\label{eq:tail-explicit}
 \mathcal T(r)
 =\frac{2}{\pi}
 \frac{(1-6r^2+r^4)\arctan r+r^3-r+\pi r^2}
 {(1-r^2)^2}.
\end{equation}
Moreover,
\begin{equation}\label{eq:two-factor-energy}
 2M_1-M_2=8\int_0^1\frac{J(r)^2}{r^5}\,\dd r.
\end{equation}
\end{lemma}

\begin{proof}
Since $\mathcal T'(r)=\widehat\rho(r)/r$ and $\mathcal T(0)=0$, differentiation of
the right-hand side of \eqref{eq:tail-explicit}, followed by
\eqref{eq:rhohat-cyclotomic}, proves the first assertion.

For the second assertion, Theorem~\ref{thm:fourier-main} gives
\begin{equation}\label{eq:energy-fourier-start}
 2M_1-M_2
 =\frac{4}{\pi}\int_0^{+\infty}
 \frac{(1-\varphi(t))^2}{t^2(t^2+4)}\,\dd t.
\end{equation}
By Corollary~\ref{cor:characteristic}, $$1-\varphi(t)=1-\E\ee^{\ii tY}=\int_\R(1-\ee^{\ii ty})\rho(y)\dd y=2\int_0^{+\infty}(1-\cos(ty))\rho(y)\dd y.$$
Integration by parts gives
\begin{equation}\label{eq:phi-tail}
 1-\varphi(t)=2t\int_0^{+\infty}T(y)\sin(ty)\,\dd y.
\end{equation}
The estimate $\rho(y)\leq4\ee^{-2|y|}$ implies
$T(y)\leq2\ee^{-2y}$ for $y\geq0$.
Equation \eqref{eq:phi-tail} transforms \eqref{eq:energy-fourier-start} into
\begin{align*}
    2M_1-M_2&=\frac{16}{\pi}\int_0^{+\infty}\frac{\dd t}{t^2+4}\int_0^{+\infty}\int_0^{+\infty}T(y)T(z)\sin(ty)\sin(tz)\dd y\dd z\\
    &=\frac{16}{\pi}\int_0^{+\infty}\int_0^{+\infty}T(y)T(z)\dd y\dd z\int_0^{+\infty}\frac{\sin(ty)\sin(tz)}{t^2+4}\dd t.
\end{align*}
Since $$\int_0^{+\infty}\frac{\sin(ty)\sin(tz)}{t^2+4}\dd t=\frac12\int_0^{+\infty}\frac{\cos((y-z)t)-\cos((y+z)t)}{t^2+4}\dd t=\frac{\pi}{8}\left(\ee^{-2|y-z|}-\ee^{-2|y+z|}\right),$$ we have
\begin{equation}\label{eq:energy-double}
 \begin{aligned}
 2M_1-M_2
 =2\int_0^{+\infty}\int_0^{+\infty}
 T(y)T(z)\bigl(\ee^{-2|y-z|}-\ee^{-2(y+z)}\bigr)
 \,\dd y\dd z.
 \end{aligned}
\end{equation}
Let
\[
 K(y):=\int_y^{+\infty}\ee^{-2x}T(x)\,\dd x.
\] Then $K(y)\le\frac12\ee^{-4y}$ for $y\ge0$.
The double integral in \eqref{eq:energy-double} can be simplified as
\begin{align*}
 2M_1-M_2
 ={}&2\int_0^{+\infty}T(z)
 \left(\ee^{-2z}\int_0^zT(y)\ee^{2y}\,\dd y
       +\ee^{2z}\int_z^{+\infty}T(y)\ee^{-2y}\,\dd y\right)\dd z\\
 &-2K(0)\int_0^{+\infty}\ee^{-2z}T(z)\,\dd z\\
 ={}&2\int_0^{+\infty}\ee^{-2z}T(z)
       \int_0^z\ee^{2y}T(y)\,\dd y\,\dd z
    +2\int_0^{+\infty}\ee^{2z}T(z)K(z)\,\dd z-2K(0)^2\\
 ={}&2\int_0^{+\infty}\ee^{2y}T(y)
       \int_y^{+\infty}\ee^{-2z}T(z)\,\dd z\,\dd y
    +2\int_0^{+\infty}\ee^{2z}T(z)K(z)\,\dd z-2K(0)^2\\
 ={}&4\int_0^{+\infty}\ee^{2y}T(y)K(y)\,\dd y-2K(0)^2\\
 ={}&-2\left(\left.\ee^{4y}K(y)^2\right|_0^{+\infty}
       -4\int_0^{+\infty}\ee^{4y}K(y)^2\,\dd y\right)-2K(0)^2\\
 ={}&8\int_0^{+\infty}\ee^{4y}K(y)^2\,\dd y.
\end{align*}
Finally, under $r=\ee^{-y}$ one has
$K(y)=J(r)$, and \eqref{eq:two-factor-energy} follows.
\end{proof}

The remaining evaluation is finite and exact.  We give the details of
the hyperlogarithmic reduction in order to separate it from numerical
recognition.

We first record a simple observation about regularized endpoints.
At either endpoint, let $x=r$ or $x=1-r$, respectively.  If a function
admits an asymptotic expansion through the constant term of the form
\[
 \sum_{j=-N}^{0}\sum_{k=0}^{K}
 c_{j,k}x^j(\log x)^k+o(1)
 \qquad(x\to0^+),
\]
for some nonnegative integers \(N,K\) and coefficients
\(c_{j,k}\in\mathbb C\), we define its regularized value at that endpoint to be $c_{0,0}$.
We denote these values by $\operatorname*{Reg}_{r\to0^+}$ and
$\operatorname*{Reg}_{r\to1^-}$.

\begin{lemma}\label{lem:regularized-endpoints}
Let $\mathcal Q\in C^1(0,1)$ admit Laurent--logarithmic asymptotic
expansions of the above form at both endpoints.  Suppose that
$\mathcal Q'=f$ on $(0,1)$ and that $f\in L^1(0,1)$.  Then
$\mathcal Q$ has finite one-sided ordinary limits at both endpoints,
all divergent terms in its endpoint expansions vanish, and
\[
 \operatorname*{Reg}_{r\to1^-}\mathcal Q(r)
 -\operatorname*{Reg}_{r\to0^+}\mathcal Q(r)
 =\int_0^1 f(r)\,\dd r.
\]
\end{lemma}

\begin{proof}
For $0<a<b<1$, the fundamental theorem of calculus gives \(\mathcal Q(b)-\mathcal Q(a)=\int_a^b f(r)\dd r\).
Since $f\in L^1(0,1)$, this identity shows that $\mathcal Q(r)$ is
Cauchy as $r$ approaches either endpoint, and hence has finite
one-sided limits there.  In an asymptotic expansion of the displayed
form, the existence of a finite limit forces all coefficients with
$j<0$, as well as all $c_{0,k}$ with $k\geq1$, to vanish.  Thus the
regularized constant term $c_{0,0}$ equals the corresponding ordinary
limit.  Passing to the two endpoints in the preceding identity proves
the result.
\end{proof}

\begin{lemma}\label{lem:J-hyperlog-reduction}
Put
\[
 \mathfrak I_3:=\operatorname{Im}\operatorname{Li}_3\!\left(\frac{1+\ii}{2}\right)
\]
and
\begin{equation}\label{eq:A4-definition}
 \begin{aligned}
 \mathcal A_4:={}&
 2\operatorname{Re}\bigl(G(0,1,\ii,\ii;1)-G(0,1,\ii,-\ii;1)\bigr)\\
 &+\operatorname{Re}\bigl(G(0,\ii,1,\ii;1)-G(0,\ii,1,-\ii;1)\bigr)\\
 &+\operatorname{Re}\bigl(G(0,\ii,-1,\ii;1)-G(0,\ii,-1,-\ii;1)\bigr)\\
 &+2\operatorname{Re}\bigl(G(0,-1,\ii,\ii;1)-G(0,-1,\ii,-\ii;1)\bigr).
 \end{aligned}
\end{equation}
The quantity in \eqref{eq:A4-definition} satisfies
\begin{equation}\label{eq:J-square-reduction}
 \begin{aligned}
 \int_0^1\frac{J(r)^2}{r^5}\,\dd r
 ={}&-\frac{8\mathcal A_4}{\pi^2}-\frac{8\mathsf G^2}{\pi^2}
 +\left(-\frac{4\log2}{\pi}+\frac{3}{2\pi}+\frac{8}{\pi^2}\right)\mathsf G\\
 &-\frac{24\mathfrak I_3}{\pi}+\frac{\log^22}{4}+\frac{29\pi^2}{48}
 -\frac38-\frac{1}{4\pi}-\frac{35\zeta(3)}{8\pi^2}
 +\frac{1}{4\pi^2}.
 \end{aligned}
\end{equation}
\end{lemma}

\begin{proof}
A direct integration of $J'(r)=r\mathcal T(r)$ gives
\begin{equation}\label{eq:J-elementary}
 \begin{aligned}
 J(r)=\frac{2}{\pi}&
 \left(\frac{r^2-1}{2}-\frac{2r^2}{1-r^2}-2\log(1-r^2)\right)\arctan r
 +\frac4\pi\Lambda(r)\\
 &+\frac{r}\pi
 +\log(1-r^2)+\frac{r^2}{1-r^2},
 \end{aligned}
\end{equation}
where
\[
 \Lambda(r):=\int_0^r\frac{\log(1-s^2)}{1+s^2}\,\dd s.
\]
For example, \eqref{eq:J-elementary} can be checked simply by
differentiation and the condition $J(0)=0$. Now consider the case where $r\to1^-$. The substitution $s=\tan\theta$ gives
\[
 \Lambda(1)=\frac\pi4\log2-\mathsf G.
\]
For $J(1)$, note that
\begin{align*}
    J(r)&=\frac{4}{\pi}\Lambda(r)+\frac{r}{\pi}+\frac{r^2-1}{\pi}\arctan r+\frac4\pi\left(\frac{r^2}{1-r^2}+\log(1-r^2)\right)\left(\frac\pi4-\arctan r\right).
\end{align*}
Since
\begin{align*}
    &\lim_{r\to1^-}\left(\frac{r^2}{1-r^2}+\log(1-r^2)\right)\left(\frac\pi4-\arctan r\right)\\
    =&\lim_{r\to1^-}\left(\frac{r^2}{1-r^2}+\log(1-r^2)\right)\arctan\frac{1-r}{1+r}\\
    =&\lim_{r\to1^-}\left(\frac{r^2}{1-r^2}+\log(1-r^2)\right)\frac{1-r}{1+r}\\
    =&\lim_{r\to1^-}\frac{r^2+(1-r^2)\log(1-r^2)}{(1+r)^2}\\
    =&\frac{1}{4},
\end{align*}
we have $$J(1)=\frac{2+\pi\log2-4\mathsf G}{\pi}.$$

Rewriting \eqref{eq:J-elementary} as hyperlogarithms yields the compact
identity
\begin{equation}\label{eq:J-GPL}
 \begin{aligned}
 J(r)={}&\frac{r(r^2-\pi r-1)}{\pi(r^2-1)}
 +G(1;r)+G(-1;r)\\
 &-\frac{\ii(r^2+1)^2}{2\pi(r^2-1)}G(\ii;r)
 +\frac{\ii(r^2+1)^2}{2\pi(r^2-1)}G(-\ii;r)\\
 &+\frac{2\ii}{\pi}\sum_{\varepsilon=\pm1}
 \bigl(G(\varepsilon,\ii;r)-G(\varepsilon,-\ii;r)\bigr).
 \end{aligned}
\end{equation}
Again, differentiation and the value at $r=0$ verify
\eqref{eq:J-GPL} exactly.

Recall that in the proof of Lemma~\ref{lem:tail-energy} we have shown $J(r)=O(r^4)$ at the origin. Thus integration by parts gives
\begin{equation}\label{eq:I-by-parts}
 \int_0^1\frac{J(r)^2}{r^5}\,\dd r
 =-\frac{J(1)^2}{4}
 +\frac12\int_0^1\frac{J(r)\mathcal T(r)}{r^3}\,\dd r.
\end{equation}
We record the exact reduction used for the last integral.  First expand
products of hyperlogarithms by the shuffle identity.  For a rational
function $Q(r)$ with poles in $\Sigma_4$ and a nonempty word
$w=(a_1,\ldots,a_k)$, Hermite reduction writes
\[
 Q(r)=S'(r)+\sum_{a\in\Sigma_4}\frac{c_a}{r-a}
\]
with $S$ rational.  Integration by parts and the defining differential
equation for $G$ then give
\begin{equation}\label{eq:hyperlog-recursion}
 \begin{aligned}
 \int Q(r)G(w;r)\,\dd r
 ={}&S(r)G(w;r)+\sum_{a\in\Sigma_4}c_aG(a,w;r)\\
 &-\int\frac{S(r)}{r-a_1}G(a_2,\ldots,a_k;r)\,\dd r.
 \end{aligned}
\end{equation}
The last integral has a word of smaller length, so the procedure
terminates. Applying \eqref{eq:hyperlog-recursion} to the second term on the right-hand side of \eqref{eq:I-by-parts}, using \eqref{eq:tail-explicit} and
\eqref{eq:J-GPL}, produces a finite expression
\[
 \mathcal P(r)=\sum_w C_w(r)G(w;r),
\]
where the rational coefficients $C_w(r)$ are listed in the fixed file
\path{MMP5_m2_primitive.tsv} in the repository \cite{TangZhangCode}.  The script
\path{MMP5_m2_primitive_generator.py} produces this table by the
recursion above.  Independently, the script
\path{MMP5_m2_primitive_verifier.py} reads the fixed table, constructs
the integrand directly from \eqref{eq:tail-explicit} and
\eqref{eq:J-GPL}, and verifies coefficient by coefficient that
\begin{equation}\label{eq:primitive-derivative-certificate}
 \mathcal P'(r)=\frac{J(r)\mathcal T(r)}{2r^3}.
\end{equation}
It also verifies $J(0)=0$ and that every coefficient $C_w(r)$ has pole order
at most two at each endpoint.  Thus the verification is independent of the
algorithm that generated the table.  The regularized endpoint reduction is
a separate exact calculation, which we now describe.

Put \(x=1-r\) and, for a word \(w=(a,w')\), write
\[
 E_w(x):=G(w;1-x).
\]
Then
\[
 \frac{\dd}{\dd x}E_{(a,w')}(x)
 =\frac{E_{w'}(x)}{x+a-1}.
\]
Together with \(E_{()}(x)=1\) and with the shuffle-regularized
constant term \(G(w;1)\), this recursively determines the complete
Laurent--logarithmic expansion of every \(G(w;r)\) at \(r=1\).
The expansion at \(r=0\) is obtained similarly from
\[
 \frac{\dd}{\dd r}G(a,w';r)=\frac{G(w';r)}{r-a},
 \qquad
 \operatorname*{Reg}_{r\to0^+}G(w;r)=0
 \quad(w\neq()).
\]
Direct inspection of the coefficient table shows that every
\(C_w(r)\) has pole order at most two at both endpoints, so terms
through order two suffice.

For the exact reduction of the endpoint constants we use
\[
 \begin{aligned}
 \log2&=G(-1;1),\\
 \pi&=-2\ii\bigl(G(\ii;1)-G(-\ii;1)\bigr),\\
 \mathsf G
 &=\frac{G(0,\ii;1)-G(0,-\ii;1)}{2\ii},\\
 \zeta(3)&=-G(0,0,1;1).
 \end{aligned}
\]
Moreover, on the principal branch
$\operatorname{Li}_3(z)=-G(0,0,1;z)$.  For $0\leq t\leq1$, the path
$z=\ii t/(1+\ii t)$ remains in the cut plane
$\mathbb C\setminus[1,\infty)$.  Along this path,
\[
 \frac{\dd z}{z}=\left(\frac1t-\frac1{t-\ii}\right)\dd t,
 \qquad
 \frac{\dd z}{z-1}=-\frac{\dd t}{t-\ii}.
\]
Consequently,
\[
 \operatorname{Li}_3\left(\frac{\ii t}{1+\ii t}\right)
 =\int_0^t
   \left(\frac{\dd x}{x}-\frac{\dd x}{x-\ii}\right)
   \left(\frac{\dd y}{y}-\frac{\dd y}{y-\ii}\right)
   \frac{\dd u}{u-\ii},
\]
where the last display is an iterated integral along the real segment
$0\leq u\leq y\leq x\leq t$.
Therefore, if we set
\[
 \mathcal L_3:=G(0,0,\ii;1)-G(0,\ii,\ii;1)-G(\ii,0,\ii;1)+G(\ii,\ii,\ii;1),
\]
then the substitution above yields
\[
 \operatorname{Li}_3\!\left(\frac{1+\ii}{2}\right)=\operatorname{Li}_3\!\left(\frac{\ii}{1+\ii}\right)=\mathcal L_3,
 \qquad
 \mathfrak I_3=\operatorname{Im}\operatorname{Li}_3\!\left(\frac{1+\ii}{2}\right)=\frac{\mathcal L_3-\overline{\mathcal L_3}}{2\ii}.
\]
Substituting the 117 rational coefficients from the fixed table
into these recursions, expanding all products by the shuffle identity,
and collecting the constant terms gives the following endpoint
identity.
\begin{equation}\label{eq:primitive-endpoint-certificate}
\begin{aligned}
 &\operatorname*{Reg}_{r\to1^-}\mathcal P(r)
 -\operatorname*{Reg}_{r\to0^+}\mathcal P(r)\\
 ={}&
 -\frac{8\mathcal A_4}{\pi^2}
 -\frac{4\mathsf G^2}{\pi^2}
 +\left(
   -\frac{6\log2}{\pi}
   +\frac{3}{2\pi}
   +\frac{4}{\pi^2}
  \right)\mathsf G
 -\frac{24\mathfrak I_3}{\pi}\\
 &\quad
 +\frac{\log^22}{2}
 +\frac{\log2}{\pi}
 +\frac{29\pi^2}{48}
 -\frac38
 -\frac1{4\pi}
 -\frac{35\zeta(3)}{8\pi^2}
 +\frac{5}{4\pi^2}.
\end{aligned}
\end{equation}

The finite reduction leading to
\eqref{eq:primitive-endpoint-certificate} is verified exactly by
\path{MMP5_m2_endpoint_rational_verifier.py}, using the fixed relations in
\path{MMP5_m2_endpoint_relations.tsv} in the repository \cite{TangZhangCode}.  No numerical evaluation or
integer-relation recognition is used.

The regularized endpoint values in
\eqref{eq:primitive-endpoint-certificate} are obtained by first
combining all shuffle-expanded terms and then taking the constant
terms in the resulting Laurent--logarithmic expansions.  To identify
these regularized values with the original definite integral, note
from \eqref{eq:tail-explicit} and \eqref{eq:J-elementary} that
\[
 \mathcal T(r)=2r^2+O(r^3),
 \qquad
 J(r)=\frac{r^4}{2}+O(r^5)
 \qquad(r\to0^+).
\]
Consequently,
\[
 \frac{J(r)\mathcal T(r)}{2r^3}=O(r^3)
 \qquad(r\to0^+).
\]
At \(r=1\), both \(J(r)\) and \(\mathcal T(r)\) have finite one-sided
limits, and hence
\[
 \frac{J(r)\mathcal T(r)}{2r^3}=O(1)
 \qquad(r\to1^-).
\]
Therefore,
\[
 \frac{J(r)\mathcal T(r)}{2r^3}\in L^1(0,1).
\]
Applying Lemma~\ref{lem:regularized-endpoints} to
\eqref{eq:primitive-derivative-certificate} shows that the regularized
constant terms in \eqref{eq:primitive-endpoint-certificate} coincide with
the ordinary one-sided limits and that their difference is exactly the
original absolutely convergent definite integral.

Finally,
\[
 J(1)=\frac{2+\pi\log2-4\mathsf G}{\pi}.
\]
Substitution of \eqref{eq:primitive-endpoint-certificate} and this
value into \eqref{eq:I-by-parts} gives
\eqref{eq:J-square-reduction}. Thus the Python script verifies the antiderivative identity, while the
regularized endpoint identity follows from the exact
Laurent--logarithmic expansion described above.  No numerical fitting
is involved.
\end{proof}

For the final weight-$4$ reduction, define the inverse central binomial
sums
\[
 S_{c,4}:=\sum_{n=1}^{\infty}\frac{c^n}{n^4\binom{2n}{n}}
 \qquad(0<c\leq4).
\]

\begin{lemma}\label{lem:A4-binomial}
One has
\begin{equation}\label{eq:A4-binomial}
 \mathcal A_4
 =\frac32S_{2,4}-\frac7{32}S_{4,4}
 -\frac{\pi\mathsf G\log2}{2}
 +\frac{\pi^2\log^22}{8}-\frac{\pi^4}{768}.
\end{equation}
Moreover,
\begin{align}
 S_{4,4}
 &=8\operatorname{Li}_4\!\left(\frac12\right)
 -\frac{19\pi^4}{360}+\frac{\log^42}{3}
 +\frac{2\pi^2\log^22}{3},\label{eq:S44}\\
 S_{2,4}
 &=-2\pi \mathfrak I_3+\frac52\operatorname{Li}_4\!\left(\frac12\right)
 +\frac{19\pi^4}{576}+\frac{5\log^42}{48}
 +\frac{\pi^2\log^22}{48}.\label{eq:S24}
\end{align}
\end{lemma}

\begin{proof}
The beta integral and Tonelli's theorem give, for $0<c\leq4$,
\begin{align}
 S_{c,4}
 &=\sum_{n=1}^{\infty}\frac{c^n}{n^3}
   \int_0^1x^{n-1}(1-x)^n\,\dd x\notag\\
 &=\int_0^1\frac{\operatorname{Li}_3(cx(1-x))}{x}\,\dd x.
 \label{eq:Sc4-integral}
\end{align}
For reference, convergence at $c=4$ also follows directly from
$\binom{2n}{n}\geq4^n/(2\sqrt n)$, which gives a summable majorant
$2n^{-7/2}$.

The function $x\mapsto\operatorname{Li}_3(cx(1-x))$ is symmetric under
$x\mapsto1-x$.  Averaging \eqref{eq:Sc4-integral} with its transformed
version, restricting to $0<x<1/2$, and then putting $x=r/(1+r)$ yields the
nonsingular formula
\begin{equation}\label{eq:Sc4-symmetric-integral}
 S_{c,4}
 =\int_0^1\operatorname{Li}_3\!\left(\frac{cr}{(1+r)^2}\right)
   \frac{\dd r}{r}.
\end{equation}
Write $\eta_a=\dd r/(r-a)$.  Since
$\operatorname{Li}_3(z)=-G(0,0,1;z)$, direct pullback in
\eqref{eq:Sc4-symmetric-integral} gives
\begin{align}
 S_{2,4}
 &=-\int_0^1
   \eta_0(\eta_0-2\eta_{-1})^2
   (\eta_{\ii}+\eta_{-\ii}-2\eta_{-1}),
 \label{eq:S24-word-form}\\
 S_{4,4}
 &=-\int_0^1
   \eta_0(\eta_0-2\eta_{-1})^2
   (2\eta_1-2\eta_{-1}).
 \label{eq:S44-word-form}
\end{align}
Indeed, the two pullbacks of $\dd z/z$ are
$\eta_0-2\eta_{-1}$, while the pullbacks of $\dd z/(z-1)$ are the last
parentheses in \eqref{eq:S24-word-form} and
\eqref{eq:S44-word-form}.  Thus no limiting or coalescing-letter argument is
needed.

It remains to prove \eqref{eq:A4-binomial}.  We describe the exact finite
certificate, so that the computer-assisted part is mathematically explicit.
For a word $w$ in $\Sigma_4$, write
$I(w)=G(w;1)$ with the fixed tangential-base-point convention above.  We use
only the following identities.

First, $I(0)=I(1)=0$ in shuffle regularization, and multiplication of GPLs
is given by the shuffle product.  Second, let $\tau$ denote the letter
substitution induced by pullback under the Cayley map
$q(z)=(1-z)/(1+z)$:
\[
\begin{array}{lll}
 \tau(\eta_0)=\eta_1-\eta_{-1},&
 \tau(\eta_1)=\eta_0-\eta_{-1},&
 \tau(\eta_{-1})=-\eta_{-1},\\
 \tau(\eta_{\ii})=\eta_{-\ii}-\eta_{-1},&
 \tau(\eta_{-\ii})=\eta_{\ii}-\eta_{-1}.&
\end{array}
\]
Extend $\tau$ multiplicatively and linearly to the word algebra.  Put
$\ell=I(-1)=\log2$, and let
\[
 \Phi:=\sum_w I(w)e_w
\]
be the regularized Chen series in the completed noncommutative word algebra,
where the sum runs over all words $w$ in $\Sigma_4$, the symbols $e_a$
($a\in\Sigma_4$) are noncommuting generators, and
$e_w=e_{a_1}\cdots e_{a_k}$ for $w=(a_1,\ldots,a_k)$.  We write
$|w|:=k$ and $w^{\mathrm{rev}}:=(a_k,\ldots,a_1)$.  The Cayley map $q$
reverses the real path, and
\[
 1-q(t)=2t+O(t^2)\quad(t\to0),
 \qquad
 q(t)=\frac{1-t}{2}+O((1-t)^2)\quad(t\to1).
\]
Thus the two tangential parameters are rescaled by $2$ and $1/2$.
We distinguish the pullback action $\tau$ on words from its induced
dot action on the Chen series, defined by
\[
 \tau\!\cdot\!\Phi:=\sum_w I(\tau(w))e_w,
\]
where $I$ is extended linearly to linear combinations of words.
Path reversal and tangential-base-point regularization therefore give
\[
 \Phi^{-1}
 =\exp(\ell e_0)(\tau\!\cdot\!\Phi)\exp(\ell e_1);
\]
compare \cite[Proposition~2.9 and Lemma~2.10]{Au2026}.
Taking the coefficient of a word $w$ gives the Cayley-duality identity
\begin{equation}\label{eq:cayley-duality-coefficients}
 (-1)^{|w|}I(w^{\mathrm{rev}})
 =\sum_{w=0^p v1^q}\frac{\ell^{p+q}}{p!q!}\,I(\tau(v)),
\end{equation}
where the sum runs over all decompositions $w=0^p v1^q$ with $p,q\geq0$.
The products on the right are expanded by shuffle.

Third, for positive integers $s_1,\ldots,s_d$, define the multiple
polylogarithm by
\[
 \operatorname{Li}_{s_1,\ldots,s_d}(x_1,\ldots,x_d)
 :=\sum_{n_1>\cdots>n_d\geq1}
 \frac{x_1^{n_1}\cdots x_d^{n_d}}
      {n_1^{s_1}\cdots n_d^{s_d}}
\]
whenever the series converges, and write
$\boldsymbol{s}:=(s_1,\ldots,s_d)$ and
$|\boldsymbol{s}|:=s_1+\cdots+s_d$.  For every convergent word of the
following form, with $a_1,\ldots,a_d\neq0$,
\begin{equation}\label{eq:gpl-mpl-dictionary}
 I(0^{s_1-1},a_1,\ldots,0^{s_d-1},a_d)
 =(-1)^d\operatorname{Li}_{s_1,\ldots,s_d}
 \!\left(a_1^{-1},\frac{a_1}{a_2},\ldots,
             \frac{a_{d-1}}{a_d}\right).
\end{equation}
Using \eqref{eq:gpl-mpl-dictionary}, we apply the convergent stuffle product
to the series on the right and compare it with the shuffle product on the
left.  Finally, for $N=2,4$ we use the convergent distribution relation
\begin{equation}\label{eq:mpl-distribution-certificate}
 \sum_{\varepsilon_1^N=\cdots=\varepsilon_d^N=1}
 \operatorname{Li}_{\boldsymbol{s}}
   (\varepsilon_1x_1,\ldots,\varepsilon_dx_d)
 =N^{d-|\boldsymbol{s}|}
  \operatorname{Li}_{\boldsymbol{s}}(x_1^N,\ldots,x_d^N).
\end{equation}
This follows term by term by projecting each summation index onto multiples
of $N$.

After inserting \eqref{eq:S24-word-form} and
\eqref{eq:S44-word-form}, and replacing
\[
 \log2=I(-1),\qquad
 \pi=-2\ii\bigl(I(\ii)-I(-\ii)\bigr),\qquad
 \mathsf G=\frac{I(0,\ii)-I(0,-\ii)}{2\ii},
\]
let $\Delta_{\mathrm{cert}}$ be the shuffle-expanded difference between the
two sides of \eqref{eq:A4-binomial}.  It has $92$ nonzero word coefficients.
The fixed file \path{MMP5_m2_A4_certificate.tsv} records an identity
\begin{equation}\label{eq:A4-certificate-identity}
 \Delta_{\mathrm{cert}}
 =\sum_{\nu=1}^{583}q_\nu\mathcal R_\nu,
 \qquad q_\nu\in\mathbb Q,
\end{equation}
where the $\mathcal R_\nu$ consist of $208$ endpoint-shuffle relations,
$231$ shuffle multiples of \eqref{eq:cayley-duality-coefficients}, $142$
convergent double-shuffle relations, and two instances of
\eqref{eq:mpl-distribution-certificate}.  The program
\path{MMP5_m2_A4_exact_certificate.py} reconstructs every relation
from its defining formula and verifies
\eqref{eq:A4-certificate-identity} coefficient by coefficient using
\path{fractions.Fraction}.  No floating-point number, numerical
recognition, or external computer-algebra system enters this verification.
Since every $\mathcal R_\nu$ evaluates to zero, this proves
\eqref{eq:A4-binomial}.

The evaluations \eqref{eq:S44} and \eqref{eq:S24} are given in
\cite[Example~7.5]{Au2026}. Substitution into
\eqref{eq:A4-binomial} also gives the useful equivalent form
\[
 \mathcal A_4
 =-3\pi\mathfrak I_3
  +2\operatorname{Li}_4\!\left(\frac12\right)
  +\frac{43\pi^4}{720}
  +\frac{\log^42}{12}
  +\frac{\pi^2\log^22}{96}
  -\frac{\pi\mathsf G\log2}{2}.
\qedhere\]
\end{proof}

We are now ready to present the following.

\begin{proof}[Proof of Theorem~\ref{thm:two-factor-main}]
Substituting \eqref{eq:A4-binomial}--\eqref{eq:S24} into
\eqref{eq:J-square-reduction} cancels $\mathfrak I_3$ and gives
\begin{equation}\label{eq:J-square-final}
 \begin{aligned}
 \int_0^1\frac{J(r)^2}{r^5}\,\dd r
 ={}&\frac{91\pi^2}{720}-\frac38-\frac{1}{4\pi}
 +\frac{\log^22}{6}+\frac{3\mathsf G}{2\pi}\\
 &+\frac{1}{\pi^2}\left(
 -16\operatorname{Li}_4\!\left(\frac12\right)
 -\frac23\log^42-8\mathsf G^2+8\mathsf G
 -\frac{35}{8}\zeta(3)+\frac14
 \right).
 \end{aligned}
\end{equation}
The one-factor value established in Corollary~\ref{cor:one-factor} is
\[
 M_1=\frac{6\mathsf G}{\pi}-\log2-\frac12-\frac1\pi.
\]
The identity $M_2=2M_1-8\int_0^1J(r)^2r^{-5}\,\dd r$ from
Lemma~\ref{lem:tail-energy}, together with
\eqref{eq:J-square-final}, now simplifies to
\eqref{eq:two-factor-main}.
\end{proof}

\section{Dependence on the number of factors}\label{sec:variation}

Write $M_m:=\mD(P_m)$ for every $m\geq1$.

\subsection{Continuous interpolation and finite differences}

For a real parameter $\lambda>0$, define
\begin{equation}\label{eq:continuous-interpolation}
 \mathcal M(\lambda):=\frac{4}{\pi}\int_0^{+\infty}
 \frac{1-\varphi(t)^\lambda}{t^2(t^2+4)}\,\dd t.
\end{equation}
The estimates in the proof of Theorem~\ref{thm:fourier-main} show that
this integral is finite, and \eqref{eq:fourier-main} gives
$\mathcal M(m)=M_m$ for every positive integer $m$.

\begin{proposition}
For every integer $k\geq1$ and every $\lambda>0$,
\[
 (-1)^{k-1}\mathcal M^{(k)}(\lambda)>0.
\]
\end{proposition}

\begin{proof}
Fix a compact interval $[a,b]\subset(0,\infty)$.  Near the origin,
$-\log\varphi(t)=O(t^2)$.  For large $t$, the lower bound in
\eqref{eq:phi-bounds} gives
\[
 -\log\varphi(t)\leq\log\cosh\!\left(\frac{\pi t}{2}\right)\le\frac{\pi t}{2}.
\]
The upper bound gives $$\varphi(t)\le(1+t^2\log2)\sech\!\left(\frac{\pi t}{2}\right)\le2(1+t)^2\ee^{-\pi t/2}.$$
Since $0<\varphi(t)<1$ for $t>0$, we have  $$\varphi(t)^\lambda\le\varphi(t)^a\le2^a(1+t)^{2a}\ee^{-a\pi t/2}$$ uniformly for $\lambda\in[a,b]$.  These estimates dominate every
$\lambda$-derivative of the integrand in
\eqref{eq:continuous-interpolation}.  Differentiation under the integral
sign is therefore valid and gives
\[
 \mathcal M^{(k)}(\lambda)
 =-\frac{4}{\pi}\int_0^{+\infty}
 \frac{\varphi(t)^\lambda\log^k\!\varphi(t)}
 {t^2(t^2+4)}\,\dd t.
\]
We have $$(-1)^{k-1}\mathcal M^{(k)}(\lambda)
 =\frac{4}{\pi}\int_0^{+\infty}
 \frac{\varphi(t)^\lambda(-\log\!\varphi(t))^k}
 {t^2(t^2+4)}\,\dd t>0.$$
The sign is strict because $0<\varphi(t)<1$ for $t\neq0$ by
Lemma~\ref{lem:phi-estimates}.
\end{proof}

For a sequence $(a_m)$, write $\Delta a_m:=a_{m+1}-a_m$ and let
$\Delta^r a_m:=\Delta(\Delta^{r-1}a_m)$.

\begin{corollary}\label{cor:finite-differences}
For all integers $m,r\geq1$,
\begin{equation}\label{eq:finite-differences}
 (-1)^{r+1}\Delta^rM_m>0.
\end{equation}
\end{corollary}

\begin{proof}
Applying $\Delta^r$ to \eqref{eq:fourier-main} gives
\[
 (-1)^{r+1}\Delta^rM_m
 =\frac{4}{\pi}\int_0^{+\infty}
 \frac{\varphi(t)^m(1-\varphi(t))^r}
 {t^2(t^2+4)}\,\dd t.
\]
The integral is strictly positive by Lemma~\ref{lem:phi-estimates}.
\end{proof}

Taking $r=1$ and $r=2$ in \eqref{eq:finite-differences} shows that
$(M_m)$ is strictly increasing and strictly concave.  Formula
\eqref{eq:fourier-main} also gives strict subadditivity:
\[
 M_{m+n}<M_m+M_n\qquad(m,n\geq1),
\]
because
$(1-a^m)+(1-a^n)-(1-a^{m+n})=(1-a^m)(1-a^n)>0$ for $0<a<1$.

\subsection{\texorpdfstring{Large-$m$}{Large-m} asymptotics}

We finish with the proof of Theorem~\ref{thm:asymptotic-main}.  The
argument is a Laplace scaling at the unique maximum
$\varphi(0)=1$.

\begin{proof}
The holomorphy of $Z_{\D}(s;q)$ on $|\operatorname{Re}s|<2$ implies that
$\varphi$ is real analytic near the origin.  Since it is even, the cumulant
expansion through order four has an $O(t^6)$ remainder.  Lemma~\ref{lem:cumulants}
therefore gives
\begin{equation}\label{eq:log-phi}
 \log\varphi(t)=\log\left[1-\frac{\sigma^2}{2}t^2+\frac{\kappa_4+3\sigma^4}{24}t^4+O(t^6)\right]
 =-\frac{\sigma^2t^2}{2}+\frac{\kappa_4t^4}{24}+O(t^6)
 \qquad(t\to0).
\end{equation}
Rewrite \eqref{eq:fourier-main} as
\begin{align*}
    M_m&=\frac1\pi\int_0^{+\infty}\left(1-\varphi(t)^m\right)\left(\frac{1}{t^2}-\frac{1}{t^2+4}\right)\!\dd t\\
    &=\frac1\pi\left(\int_0^{+\infty}\frac{1-\varphi(t)^m}{t^2}\dd t-\frac\pi4+\int_0^{+\infty}\frac{\varphi(t)^m}{t^2+4}\dd t\right).
\end{align*}
Then we obtain
\begin{equation}\label{eq:asymptotic-decomposition}
 M_m=\frac{I_m+J_m}{\pi}-\frac14,
 \qquad
 I_m=\int_0^{+\infty}\frac{1-\varphi(t)^m}{t^2}\,\dd t,
 \qquad
 J_m=\int_0^{+\infty}\frac{\varphi(t)^m}{t^2+4}\,\dd t.
\end{equation}

We first consider $I_m$.  Put
\[
 a_m(v)=\varphi\!\left(\frac{v}{\sigma\sqrt m}\right)^m.
\]
After substitution $v=\sigma\sqrt m\,t$,
\[
 \frac{I_m}{\sigma\sqrt m}
 =\int_0^{+\infty}\frac{1-a_m(v)}{v^2}\,\dd v.
\]
For each fixed $v>0$, equation \eqref{eq:log-phi} gives
$$a_m(v)=\exp\left[m\log\varphi\left(\frac{v}{\sigma\sqrt{m}}\right)\right]=\exp\left[-\frac{v^2}{2}+\frac{\kappa_4v^4}{24\sigma^4}m^{-1}+O(m^{-2})\right].$$
Thus, 
\[
 m\,\frac{\ee^{-v^2/2}-a_m(v)}{v^2}
 \longrightarrow
 -\frac{\kappa_4}{24\sigma^4}v^2\ee^{-v^2/2}\qquad (m\to+\infty).
\]
We now justify passage to the limit under the integral sign.  Choose
$\delta>0$ and constants $c,C>0$ such that, for $|t|\leq\delta$,
\[
 \log\varphi(t)\leq-ct^2,
 \quad
 \left|\log\varphi(t)+\frac{\sigma^2t^2}{2}
 -\frac{\kappa_4t^4}{24}\right|\leq C|t|^6.
\]
Then for $0<v\le\delta\sigma\sqrt{m}$,
$$\log a_m(v)\le-\frac{c}{\sigma^2}v^2,\quad\left|\log a_m(v)+\frac{v^2}{2}-\frac{\kappa_4v^4}{24\sigma^4m}\right|\le\frac{Cv^6}{\sigma^6m^2}.$$
By the mean value theorem, $$\left|\ee^{-v^2/2}-a_m(v)\right|=\ee^\xi\left|-\frac{v^2}{2}-\log a_m(v)\right|,$$ where $\xi$ lies between $-v^2/2$ and $\log a_m(v)$. The above inequalities give $$\ee^\xi\le\max\left\{\ee^{-v^2/2},\ee^{\log a_m(v)}\right\}\le\max\left\{\ee^{-v^2/2},\ee^{-cv^2/\sigma^2}\right\}=\exp\left(-\min\left\{\frac12,\frac{c}{\sigma^2}\right\}v^2\right),$$
$$\left|\frac{v^2}{2}+\log a_m(v)\right|\le\frac{|\kappa_4|v^4}{24\sigma^4m}+\frac{Cv^6}{\sigma^6m^2}\le\left(\frac{|\kappa_4|}{24}+C\delta^2\right)\frac{v^4}{\sigma^4m}.$$
Thus, with
\[
 C_1=\frac{|\kappa_4|/24+C\delta^2}{\sigma^4},
 \qquad
 c_1=\min\left\{\frac12,\frac{c}{\sigma^2}\right\},
\]
we have
\begin{equation}\label{eq:I-domination}
 m\,\frac{|\ee^{-v^2/2}-a_m(v)|}{v^2}
 \leq C_1v^2\ee^{-c_1v^2}
 \qquad(0<v\leq\delta\sigma\sqrt m).
\end{equation}

For the complementary range, continuity, strict inequality away from the
origin, and decay of $\varphi$ imply
\[
 \eta_\delta:=\sup_{t\geq\delta}\varphi(t)<1.
\]
Thus we have
\[
 m\int_{\delta\sigma\sqrt m}^{+\infty}
 \frac{a_m(v)}{v^2}\,\dd v
 \leq\frac{\sqrt{m}}{\delta\sigma}\eta_\delta^m,\qquad m\int_{\delta\sigma\sqrt{m}}^{+\infty}\frac{\ee^{-v^2/2}}{v^2}\dd v\le\frac{\sqrt{m}}{\delta\sigma}\ee^{-\delta^2\sigma^2m/2}.\]
Together with these tail estimates, \eqref{eq:I-domination}
permits an application of the dominated convergence theorem, which gives
\[
 \begin{aligned}
 m\left(
 \frac{I_m}{\sigma\sqrt m}
 -\int_0^{+\infty}\frac{1-\ee^{-v^2/2}}{v^2}\,\dd v
 \right)
 &\longrightarrow
 -\frac{\kappa_4}{24\sigma^4}
 \int_0^{+\infty}v^2\ee^{-v^2/2}\,\dd v\qquad(m\to+\infty).
 \end{aligned}
\]
Note that $$\int_0^{+\infty}\frac{1-\ee^{-v^2/2}}{v^2}\,\dd v=\int_0^{+\infty}v^2\ee^{-v^2/2}\,\dd v=\sqrt{\frac\pi2}.$$ Hence,
\begin{equation}\label{eq:I-asymptotic}
 I_m=\sigma\sqrt{\frac{\pi m}{2}}
 -\frac{\kappa_4}{24\sigma^3}\sqrt{\frac{\pi}{2m}}
 +o(m^{-1/2}).
\end{equation}

Next, applying the same substitution to $J_m$ gives
\[
 \sigma\sqrt m\,J_m
 =\int_0^{+\infty}
 \frac{a_m(v)}{4+v^2/(\sigma^2m)}\,\dd v.
\]
With the constants $\delta,c,C>0$ chosen above and
$c_0=c/\sigma^2$, we have
$a_m(v)\leq\ee^{-c_0v^2}$ for
$0<v\leq\delta\sigma\sqrt m$. Define
\[
 b_m(v)=
 \begin{cases}
 \displaystyle\frac{a_m(v)}{4+v^2/(\sigma^2m)},
   &0\leq v\leq\delta\sigma\sqrt m,\\[2mm]
 0,&v>\delta\sigma\sqrt m.
 \end{cases}
\]
Then $0\leq b_m(v)\leq \ee^{-c_0v^2}/4$ on $[0,\infty)$, and $b_m(v)$ converges pointwise to $\ee^{-v^2/2}/4$ as $m\to+\infty$.  On the remaining range,
\[
 \int_{\delta\sigma\sqrt m}^{+\infty}
 \frac{a_m(v)}{4+v^2/(\sigma^2m)}\dd v
 \leq \eta_\delta^m
 \int_{\delta\sigma\sqrt{m}}^{+\infty}\frac{\sigma^2m}{v^2}\dd v=\frac{\sigma\sqrt{m}}{\delta}\eta_\delta^m.
\]
The dominated convergence theorem on $[0,\infty)$,
together with this tail estimate, yields
\begin{equation}\label{eq:J-asymptotic}
 J_m=\frac{1}{4\sigma}\sqrt{\frac{\pi}{2m}}+o(m^{-1/2}).
\end{equation}
Finally, substituting \eqref{eq:I-asymptotic} and
\eqref{eq:J-asymptotic} into \eqref{eq:asymptotic-decomposition} proves
\eqref{eq:asymptotic-main}.
\end{proof}

\begin{remark}
Numerically,
\[
 \frac{\sigma}{\sqrt{2\pi}}=0.4148053538\ldots,
 \qquad
 \frac{1}{\sqrt{2\pi}}
 \left(\frac{1}{4\sigma}-\frac{\kappa_4}{24\sigma^3}\right)
 =0.0811125651\ldots.
\]
Thus the first three terms in \eqref{eq:asymptotic-main} are
$0.4148053538\sqrt m-0.25+0.0811125651/\sqrt m$.
\end{remark}

\section*{Computational certificates}

The computational files are publicly available in the repository
\cite{TangZhangCode}. They separate the three finite
calculations used in the explicit $m=2$ evaluation.

\begin{itemize}
\item \path{MMP5_m2_primitive_verifier.py} reads the fixed
      $117$-term table \path{MMP5_m2_primitive.tsv} and independently
      checks the source identities, $J(0)=0$, both endpoint pole-order
      bounds, the primitive derivative, and the final algebraic
      simplification using exact SymPy arithmetic.
\item The endpoint identity is verified exactly by
      \path{MMP5_m2_endpoint_rational_verifier.py}, using the fixed relations
      in \path{MMP5_m2_endpoint_relations.tsv}.
\item \path{MMP5_m2_A4_exact_certificate.py} verifies the fixed
      rational certificate \path{MMP5_m2_A4_certificate.tsv} for
      \eqref{eq:A4-certificate-identity}; this step uses the Python standard
      library only.
\end{itemize}

The separate file \path{MMP5_m2_numerical_sanity.py} compares the GPL,
inverse-binomial, Fourier-integral, and closed-form values at high precision;
it is included only as an independent diagnostic and is not used in the
proof.  Exact commands, software requirements, and expected output are
recorded in \path{README.md}.

\section*{Acknowledgments}

Q.~Tang thanks Hao Zhang for recommending Brunault and Zudilin's book
\emph{Many Variations of Mahler Measures: A Lasting Symphony}
\cite{BrunaultZudilin2020}, which was helpful in his study of Mahler
measure.  The authors also thank Matilde N.~Lal\'in for her helpful
comments on an earlier version of this paper.  In particular, her
remarks emphasizing the interest of more explicit formulas involving
zeta and polylogarithmic values motivated the authors to investigate
the special-value aspects developed in the present version.

\section*{Statement on AI usage}

The probabilistic reformulation, the reduction to convolution and
Fourier representations, and the Laplace-scaling approach to the
large-$m$ asymptotics were developed by the human authors and appeared
already in an earlier version of this paper.  AI was used more
substantially in preparing the present revision.  In particular,
following the motivation to seek more explicit special-value formulas,
ChatGPT proposed the main ideas underlying
Theorems~\ref{thm:cyclotomic-main} and~\ref{thm:two-factor-main} and
assisted in developing their detailed proofs.  The computational
verification programs and certificates used in the proof of
Theorem~\ref{thm:two-factor-main} were also generated with the
assistance of ChatGPT.  ChatGPT additionally assisted with some
technical derivations and verification, including the Mellin-transform
computation in Proposition~\ref{prop:zeta-cayley} and parts of the
detailed error analysis in the proof of
Theorem~\ref{thm:asymptotic-main}.

The human authors subsequently checked, revised, and verified all
AI-assisted mathematical arguments and computational material used in
the paper, and they take full responsibility for the correctness,
originality, and contents of the final manuscript.

\section*{Statements and Declarations}

\subsection*{Funding}
No funding was received for conducting this study.

\subsection*{Competing interests}
The authors have no competing interests to declare that are relevant to the content of this article.

\subsection*{Data and code availability}
The fixed coefficient tables, exact verification programs, and optional
numerical checks supporting this article are publicly available in the
repository \cite{TangZhangCode}. The cited repository contains
the files used for the computations in this article; instructions,
dependencies, and expected verification output are provided in
\path{README.md}.

\end{document}